\documentclass{amsart}

\usepackage[colorlinks=true, urlcolor=blue, citecolor=blue, linkcolor=blue, hyperfootnotes=true]{hyperref}
\usepackage{amssymb}
\usepackage{mathrsfs}
\usepackage{dsfont}
\usepackage{enumitem}
\usepackage{aliascnt}

\newcommand\barrow{\textstyle\mathop{\rightarrow}_{}^{\hspace{-8pt}\bullet}}
\newcommand\arrowb{\textstyle\mathop{\rightarrow}_{\hspace{-8pt}\bullet}^{}}

\newcommand\arrowc{\textstyle\mathop{\rightarrow}_{\hspace{-8pt}\circ}^{}}
\newcommand\barrowc{\textstyle\mathop{\rightarrow}_{\hspace{-8pt}\circ}^{\hspace{-8pt}\bullet}}

\newcommand\carrowb{\textstyle\mathop{\rightarrow}_{\hspace{-8pt}\bullet}^{\hspace{-8pt}\circ}}

\newtheorem{thm}{Theorem}[section]

\newaliascnt{prp}{thm}
\newtheorem{prp}[prp]{Proposition}
\aliascntresetthe{prp}

\newaliascnt{cor}{thm}
\newtheorem{cor}[cor]{Corollary}
\aliascntresetthe{cor}

\newaliascnt{lem}{thm}

\aliascntresetthe{lem}

\theoremstyle{definition}

\newaliascnt{dfn}{thm}
\newtheorem{dfn}[dfn]{Definition}
\aliascntresetthe{dfn}

\newaliascnt{rmk}{thm}
\newtheorem{rmk}[rmk]{Remark}
\aliascntresetthe{rmk}

\newaliascnt{qst}{thm}

\aliascntresetthe{qst}

\newaliascnt{xpl}{thm}

\aliascntresetthe{xpl}

\numberwithin{equation}{section}

\author{Tristan Bice}
\email{tristan.bice@gmail.com}
\thanks{The author is supported by the GA\v{C}R project EXPRO 20-31529X and RVO: 67985840.}
\address{Institute of Mathematics of the Czech Academy of Sciences, Prague, Czech Republic}

\title{Distance Domains: Continuity}

\begin{document}

\begin{abstract}
We take the abstract basis approach to classical domain theory and extend it to quantitative domains.  In doing so, we provide dual characterisations of distance domains (some new even in the classical case) as well as unifying and extending previous formal ball dualities, namely the Kostanek-Waszkiewicz and Romaguero-Valero theorems.  In passing, we also show that hemimetric spaces admit a hemimetric Smyth completion precisely when they are Noetherian in a natural quantitative sense.
\end{abstract}

\maketitle

\section*{Motivation}

Classical domain theory (see \cite{GierzHofmannKeimelLawsonMisloveScott2003}) traces its origins to Scott's foundational work on lambda calculus semantics in the late 60's.  Since then, applications have been found in various fields of e.g. theoretical computer science, topology and algebra.  Beginning around the late 90's, efforts have been made to develop a quantitative extension of domain theory more suitable to metric-like structures arising in analysis (see \cite{BonsangueBreugelRutten1998}, \cite{Wagner1997} and \cite{KunziSchellekens2002}).  This is where the present paper comes in, continuing the work we began in \cite{Bice2019a}, which itself is a further development of \cite{Bice2017}.  As mentioned in the introduction to \cite{Bice2019a}, our motivation comes primarily from potential applications in Banach space and C*-algebra theory, where classical domains have also found important applications in recent years \textendash\, see \cite{Keimel2016}.

The novelty of our approach comes from considering general non-symmetric distances, functions merely satisfying the triangle inequality.  In contrast, up until now the focus has been almost exclusively on more restrictive hemimetrics.  While hemimetrics provide quantitative analogs of preorders, to truly do quantitative domain theory we also need quantitative analogs of more general transitive relations, like the all important way-below relation.  In fact in \cite[\S9]{KostanekWaszkiewicz2011}, a kind of way-below distance was defined from a hemimetric, although it was considered as something of a special case.  To get the most out of quantitative domain theory, we should be able to go the other way around, starting with some non-symmetric distance from which we then define an appropriate hemimetric.  This is the approach we focus on, thus providing a quantitative version of the `abstract basis' approach to classical domains, as seen in \cite{Keimel2016}, for example.

Another key difference in our work comes from the use of topologies arising from combinations of balls and holes.  Ball topologies have certainly been considered before, but hole topologies have been almost completely ignored (although they are mentioned briefly in \cite[Exercise 6.2.11]{Goubault2013}).  However, hole topologies are key to defining appropriate topological analogs of suprema and maxima, not to mention the fact they have also appeared in various guises as certain weak topologies on spaces of linear operators, functions and subsets.

\section*{Outline}

As mentioned above, we laid the groundwork for the present paper in \cite{Bice2019a} and will make extensive use of the notation, terminology and theory presented there.  The first section is devoted to a review of the relevant parts of \cite{Bice2019a}, although we would also encourage the reader to familiarise themselves with \cite{Bice2019a} to get a full understanding of the present paper.

Several generalizations of continuity (in the order theoretic sense) have been considered in the literature.  Our approach in \autoref{Continuity} is to simply switch the quantifiers in completeness.  We then show in \autoref{dbhcont} and \autoref{maxctsequiv} how $\mathbf{d}^\bullet_\circ$-continuity and $\mathbf{d}$-$\max$-continuity can be characterized by interpolation conditions generalizing abstract bases.

Next we introduce distance analogs of the way-below relation in \autoref{WBD}.  After discussing their basic properties, we give dual characterizations of distance domains in \autoref{Tdomaineqs} and \autoref{Rdomaineqs}.  This allows us to largely bypass the way-below construction in favour of its inverse, the lower hemimetric construction.  This duality may also be of some interest even in the classical case.  Indeed, domains are usually defined as certain kinds of posets, but here we see that they can instead be defined as certain `$\max$-complete' abstract bases.

To complete $\mathbf{d}$-$\max$-predomains to $\mathbf{d}$-$\max$-domains, we introduce Hausdorff distances in \autoref{HD}, paying particular attention to the reverse Hausdorff distance and its relation to the usual Hausdorff distance.  The completion is then obtained in \autoref{predomaincompletion}, and its universality is proved in \autoref{predomainuniversality}.  In \autoref{pdcomp} we show that $\mathbf{d}$-$\max$-predomains are precisely the $\mathbf{d}$-$\max$-bases of $\mathbf{d}$-$\max$-domains.

In order to extend this completion from the relational to the topological setting, we introduce formal balls $X_+$ in \autoref{FB}.  As a precursor we show in \autoref{contdomballs} that $\mathbf{d}^\bullet_\circ$-completeness and $\mathbf{d}^\bullet_\circ$-continuity in $X$ are equivalent to their order theoretic counterparts in $X_+$.  This yields a dual formulation of a theorem of Kostanek-Waszkiewicz which also extends the Romaguera-Valero theorem \textendash\, see \autoref{KW} and the comments after it.  Lastly, we show how to complete Smyth predomains to domains in \autoref{toppredomaincompletion}, noting in \autoref{ESmyth} that the Smyth completion coincides with the Yoneda completion iff $X$ is $\mathbf{d}$-Noetherian.

\section*{Acknowledgements}

The author would like to thank Martino Lupini for many fruitful discussions on distance domains, and for his kind hospitality while visiting Victoria University of Wellington in New Zealand in November 2019.

\section{Preliminaries}

First we summarise the most important notation and conventions from \cite{Bice2019a}.

It will be convenient to consider the category $\mathbf{GRel}$ whose objects are sets and whose morphisms are `generalised relations', namely binary functions with values in $[0,\infty]$, i.e. $\mathrm{Hom}(X,Y)=[0,\infty]^{X\times Y}$.  We extend the standard infix notation for classical relations to generalised relations, i.e. $x\mathbf{a}y=\mathbf{a}(x,y)$.  Composition in $\mathbf{GRel}$ is defined via infima, specifically, for any $\mathbf{a}\in[0,\infty]^{X\times Z}$ and $\mathbf{b}\in[0,\infty]^{Z\times Y}$,
\[x(\mathbf{a}\circ\mathbf{b})y=\inf_{z\in Z}(x\mathbf{a}z+z\mathbf{b}y).\]

A generalised relation $\mathbf{a}\in[0,\infty]^{X\times Y}$ defines a classical relation $\leq^\mathbf{a}\ \subseteq X\times Y$ by
\[x\leq^\mathbf{a}y\qquad\Leftrightarrow\qquad x\mathbf{a}y=0.\]
Conversely, every classical relation $<\ \subseteq X\times Y$ will be identified with the generalised relation $<\ \in[0,\infty]^{X\times Y}$ defined by its characteristic function given by
\[<(x,y)=\begin{cases}0&\text{if }x<y\\ \infty&\text{otherwise}.\end{cases}\]
So under this identification, any $\mathbf{a}\in[0,\infty]^{Z\times X}$ yields $\mathbf{a}\ \circ<\ \in[0,\infty]^{Z\times Y}$ given by
\[x(\mathbf{a}\ \circ<)y=\inf_{z<y}x\mathbf{a}z.\]
In particular, we can take $<\ =\ \leq^\mathbf{b}$ and consider the function $\mathbf{a}\ \circ\leq^\mathbf{b}$.  We will also have occasion to consider the slightly smaller function $\mathbf{a}\circ\Phi^\mathbf{b}$ defined by
\[\mathbf{a}\circ\Phi^\mathbf{b}=\sup_{n\in\mathbb{N}}(\mathbf{a}\circ n\mathbf{b})=\sup_{\epsilon>0}(\mathbf{a}\ \circ<^\mathbf{b}_\epsilon),\]
where $x<^\mathbf{b}_\epsilon y$ means $x\mathbf{b}y<\epsilon$.

Morphisms are ordered pointwise be default, i.e. for any $\mathbf{a},\mathbf{b}\in[0,\infty]^{X\times Y}$,
\[\mathbf{a}\leq\mathbf{b}\qquad\Leftrightarrow\qquad\forall x\in X\ \forall y\in Y\ (x\mathbf{a}y\leq x\mathbf{b}y).\]
We will also have occasion to consider the weaker uniform preorder $\precapprox$, where $\mathbf{a}\precapprox\mathbf{b}$ means that, for all $Z\subseteq X\times Y$,
\[\tag{Uniform Preorder}\inf_{(x,y)\in Z}x\mathbf{b}y=0\qquad\Rightarrow\qquad\inf_{(x,y)\in Z}x\mathbf{a}y=0.\]
Equivalently, defining $\tfrac{\mathbf{a}}{\mathbf{b}}\in[0,\infty]^{[0,\infty]}$ by $\tfrac{\mathbf{a}}{\mathbf{b}}(r)=\sup_{x\mathbf{b}y\leq r}x\mathbf{a}y$ (so $\tfrac{\mathbf{a}}{\mathbf{b}}$ is the smallest monotone function satisfying $\tfrac{\mathbf{a}}{\mathbf{b}}(x\mathbf{b}y)\geq x\mathbf{a}y$) we can define/characterise $\precapprox$ by
\[\mathbf{a}\precapprox\mathbf{b}\qquad\Leftrightarrow\qquad\lim_{r\rightarrow0}\tfrac{\mathbf{a}}{\mathbf{b}}(r)=0.\]

We call $\mathbf{d}\in[0,\infty]^{X\times X}$ a \emph{distance} if $\mathbf{d}$ satisfies the triangle inequality
\[\tag{$\triangle$}\label{triangle}\mathbf{d}\leq\mathbf{d}\circ\mathbf{d}.\]
Given $<\ \subseteq X\times X$ (again identified with its characteristic function) \eqref{triangle} becomes $<\circ<\ \subseteq\ <$, which is just transitivity, i.e. distances generalise transitive relations.  We call a distance $\mathbf{d}$ a \emph{hemimetric} if $\leq^\mathbf{d}$ is reflexive and hence a preorder, while we call $\mathbf{d}$ a \emph{quasimetric} if $\leq^\mathbf{d}$ is also antisymmetric and hence a partial order.  For any $r\in\mathbb{R}$, let $r_+=r\vee0$.  From any $\mathbf{d}\in[0,\infty]^{X\times Y}$, we can define the upper and lower hemimetrics $\overline{\mathbf{d}}\in[0,\infty]^{X\times X}$ and $\underline{\mathbf{d}}\in[0,\infty]^{Y\times Y}$ by
\begin{align}
\tag{Upper Hemimetric}x\overline{\mathbf{d}}z&=\sup_{y\in Y}(x\mathbf{d}y-z\mathbf{d}y)_+.\\
\tag{Lower Hemimetric}z\underline{\mathbf{d}}y&=\sup_{x\in X}(x\mathbf{d}y-x\mathbf{d}z)_+.
\end{align}

\begin{center}
\textbf{From now on, we assume $\mathbf{d}$ and $\mathbf{e}$ are distances on a set $X$.}
\end{center}

As in classical domain theory, directed subsets and their minimal upper bounds play a fundamental role.  Specifically, we call $Y\subseteq X$ \emph{$\mathbf{d}$-directed} if
\[\tag{$\mathbf{d}$-directed}\inf_{y\in Y}\sup_{z\in F}z\mathbf{d}y=0,\]
for all $F\in\mathcal{F}(Y)=\{G\subseteq Y:G\text{ is finite}\}$.  Note that $\mathbf{d}$-directed subsets are necessarily non-empty, as we take $\inf\emptyset=\infty$.  Define functions $y\mathbf{d}$ and $\mathbf{d}z$ by
\[y\mathbf{d}(z)=y\mathbf{d}z=\mathbf{d}z(y).\]
For any $Z\subseteq Y$, we also define functions $Z\mathbf{d}$ and $\mathbf{d}Z$ by
\[Z\mathbf{d}=\sup_{z\in Z}z\mathbf{d}\qquad\text{and}\qquad\mathbf{d}Z=\inf_{z\in Z}\mathbf{d}z.\]
Then $\mathbf{d}$-directedness can be expressed as $(F\mathbf{d})Y=0$, for all $F\in\mathcal{F}(Y)$.

We will also have occasion to deal with more general \emph{$\mathbf{d}$-final} $Y\subseteq X$, meaning that $x\mathbf{d}Y=0$, for all $x\in Y$.  We say that $Y$ is \emph{$\mathbf{d}$-initial} if $Y$ is $\mathbf{d}^\mathrm{op}$-final (in domain theory, final and initial subsets would often be called `round').  If we let $\mathbf{0}$ denote the zero hemimetric and consider $\mathcal{F}\mathbf{d}\in[0,\infty]^{\mathcal{F}(X)\times X}$ defined by $Y(\mathcal{F}\mathbf{d})x=Y\mathbf{d}x$ then the entirety of $X$ being $\mathbf{d}$-directed/final/initial can be expressed succinctly via composition in $\mathbf{GRel}$, specifically
\begin{align*}
\mathcal{F}\mathbf{d}\circ\mathbf{0}=\mathcal{F}\mathbf{0}\qquad&\Leftrightarrow\qquad X\text{ is $\mathbf{d}$-directed}.\\
\mathbf{d}\circ\mathbf{0}=\mathbf{0}\qquad&\Leftrightarrow\qquad X\text{ is $\mathbf{d}$-final}.\\
\mathbf{0}\circ\mathbf{d}=\mathbf{0}\qquad&\Leftrightarrow\qquad X\text{ is $\mathbf{d}$-initial}.
\end{align*}

If $x$ is an upper $\leq^\mathbf{d}$-bound of $Y\subseteq X$, which we write as $Y\leq^\mathbf{d}x$, then we call $x$ a \emph{$\mathbf{d}$-supremum} if $x\mathbf{d}\leq Y\mathbf{d}$ and a \emph{$\mathbf{d}$-maximum} if $\mathbf{d}Y\leq\mathbf{d}x$, i.e.
\begin{align*}
x=\text{$\mathbf{d}$-$\sup Y\ $}\qquad&\Leftrightarrow\qquad Y\leq^\mathbf{d}x\quad\text{and}\quad Y\mathbf{d}\geq x\mathbf{d}.\\
x=\text{$\mathbf{d}$-$\max Y$}\qquad&\Leftrightarrow\qquad Y\leq^\mathbf{d}x\quad\text{and}\quad\mathbf{d}Y\leq\mathbf{d}x.
\end{align*}
Again these generalise the usual notions for partial order relations when they are identified with their characteristic functions in $\mathbf{GRel}$.

Alternatively, we get a subtly different version of quantitative domain theory by considering nets instead of subsets and limits instead of upper bounds.  The analog of $\mathbf{d}$-directed subsets are the $\mathbf{d}$-(pre-)Cauchy nets defined by
\begin{align}
\label{pre-Cauchy}\lim_\gamma\limsup_\delta x_\gamma\mathbf{d}x_\delta=0\qquad&\Leftrightarrow\qquad(x_\lambda)\text{ is \emph{$\mathbf{d}$-pre-Cauchy}}.\\
\label{Cauchy}\lim_\gamma\sup_{\gamma\prec\delta} x_\gamma\mathbf{d}x_\delta=0\qquad&\Leftrightarrow\qquad(x_\lambda)\text{ is \emph{$\mathbf{d}$-Cauchy}}.
\end{align}
The analogs of suprema and maxima are limits in the Yoneda and Smyth topologies respectively.  To define these, we first need to generalise a couple of standard topologies defined from partial order relations.

The \emph{Alexandroff topology}, denoted by $\mathbf{d}^\bullet$, is generated by the upper balls
\[c^\bullet_r=\{x\in X:c\mathbf{d}x<r\}.\]
The \emph{lower topology}, denoted by $\mathbf{d}_\circ$, is generated by the lower holes
\[c_\circ^r=\{x\in X:c\mathbf{d}x>r\}.\]
The \emph{Smyth topology}, denoted by $\mathbf{d}^\bullet_\circ$, is the join of the Alexandroff and lower topologies, i.e. generated by both upper balls and lower holes.  Equivalently, the Smyth topology is the weakest topology making the functions $(c\mathbf{d})_{c\in X}$ continuous.

\begin{rmk}
The name comes from the fact a quasimetric space is Smyth complete in the sense of \cite[Definition 7.2.1]{Goubault2013} iff every $\mathbf{d}$-Cauchy net has a limit in the Smyth topology \textendash\, while \cite[Definition 7.2.1]{Goubault2013} uses the symmetric ball/Alexandroff topology $\mathbf{d}^\bullet_\bullet=\mathbf{d}^{\vee\bullet}$ (where $\mathbf{d}^\vee=\mathbf{d}\vee\mathbf{d}^\mathrm{op}$), $\mathbf{d}^\bullet_\bullet$-limits and $\mathbf{d}^\bullet_\circ$-limits are the same for $\mathbf{d}$-Cauchy nets in hemimetric spaces, by \cite[(8.8), (8.9), (8.10) and (8.15)]{Bice2019a}.
\end{rmk}

The \emph{upper topology} $\mathbf{d}^\circ=(\mathbf{d}^\mathrm{op})_\circ$ is generated by the upper holes
\[c^\circ_r=\{x\in X:x\mathbf{d}c>r\}.\]
The \emph{Yoneda topology} $\mathbf{d}^\circ_\circ=\mathbf{d}^\circ\vee\mathbf{d}_\circ$ is generated by both upper and lower holes.

\begin{rmk}
Again, the name here comes from the fact a quasimetric space is Yoneda complete, in the sense of \cite[Definition 7.4.1]{Goubault2013}, iff every $\mathbf{d}$-Cauchy net has a limit in the Yoneda topology.  Again, while \cite[Definition 7.4.1]{Goubault2013} uses so called $\mathbf{d}$-limits, these are the same as $\mathbf{d}^\circ_\circ$-limits for $\mathbf{d}$-Cauchy nets in hemimetric spaces, by \cite[(8.11) and (8.16)]{Bice2019a}.
\end{rmk}

We denote convergence in $\mathbf{d}^\bullet$, $\mathbf{d}_\circ$, $\mathbf{d}^\bullet_\circ$, etc. by $\barrow$, $\arrowc$, $\barrowc$, etc..  As with subsets, for any net $(x_\lambda)$, we define functions $(x_\lambda)\mathbf{d}$ and $\mathbf{d}(x_\lambda)$ by
\[(x_\lambda)\mathbf{d}=\limsup_\lambda x_\lambda\mathbf{d}\qquad\Leftrightarrow\qquad\mathbf{d}(x_\lambda)=\liminf_\lambda\mathbf{d}x_\lambda.\]
These functions can be used to characterise convergence, e.g.
\begin{align*}
x_\lambda\arrowb x\qquad&\Leftrightarrow\qquad(x_\lambda)\mathbf{d}\leq x\mathbf{d}.\\
x_\lambda\arrowc x\qquad&\Leftrightarrow\qquad\mathbf{d}(x_\lambda)\geq\mathbf{d}x.
\end{align*}
Note that when $(x_\lambda)$ is $\mathbf{d}$-Cauchy, $\limsup$ and $\liminf$ can be replaced with $\lim$.

It will also be convenient to define what it means for a subset to be below a net and vice versa.  Specifically, for any $(x_\lambda)\subseteq X$ and $Y\subseteq X$, let
\begin{align*}
(x_\lambda)\leq^\mathbf{d}Y\qquad&\Leftrightarrow\qquad x_\lambda\mathbf{d}Y\rightarrow0.\\
Y\leq^\mathbf{d}(x_\lambda)\qquad&\Leftrightarrow\qquad\,y\mathbf{d}x_\lambda\rightarrow0,\text{ for all }y\in Y.\\
Y\equiv^\mathbf{d}(x_\lambda)\qquad&\Leftrightarrow\qquad Y\leq^\mathbf{d}(x_\lambda)\leq^\mathbf{d}Y.
\end{align*}

\section{Continuity}\label{Continuity}

Our first goal is to define and examine two general quantitative notions of continuity (one using subsets and the other using nets) which extend the classical order theoretic notions of a continuous poset and an abstract basis.

To motivate these, first recall that a poset $(X,\leq)$ is \emph{continuous} (see \cite[Definition I-1.6]{GierzHofmannKeimelLawsonMisloveScott2003} or \cite[Definition 5.1.5]{Goubault2013}) if it satisfies either of the following equivalent conditions relative to the way-below relation $\ll$ defined from $\leq$.
\begin{align*}
(X,\leq)\text{ is a continuous poset}\quad\Leftrightarrow\quad&\forall x\in X\ \exists\text{ $\ll$-directed }Y\subseteq X\ (x=\text{$\leq$-$\sup Y$})\\
\Leftrightarrow\quad&\forall x\in X\ \exists\text{ $\ll$-increasing }(x_\lambda)\ (x_\lambda\xrightarrow{\leq^\circ_\circ}x).
\end{align*}

If we instead start with a transitive relation $\ll$ and replace $\leq$-suprema with $\ll$-maxima and the $\leq$-Yoneda topology by the $\ll$-Smyth topology, we get abstract bases instead (see \cite[Definition III-4.15]{GierzHofmannKeimelLawsonMisloveScott2003} or \cite[Definition 5.1.32]{Goubault2013} for the more standard interpolation definition of abstract bases, discussed further below).
\begin{align*}
(X,\ll)\text{ is an abstract basis}\quad\Leftrightarrow\quad&\forall x\in X\ \exists\text{ $\ll$-directed }Y\subseteq X\ (x=\text{$\ll$-$\max Y$})\\
\Leftrightarrow\quad&\forall x\in X\ \exists\text{ $\ll$-increasing }(x_\lambda)\ (x_\lambda\xrightarrow{\ll^\bullet_\circ}x).
\end{align*}

Accordingly, we are led to the following general quantitative notions of continuity.

\begin{dfn}\label{ctsdef}
For any topology $\mathcal{T}$ on $X$ and relation $\mathcal{R}\subseteq X\times\mathcal{P}(X)$, define
\begin{align*}
X\text{ is \emph{$\mathbf{d}$-$\mathcal{R}$-continuous}}\quad&\Leftrightarrow\quad\forall x\in X\ \exists\text{$\mathbf{d}$-directed }Y\subseteq X\ (x\mathcal{R}Y).\\
X\text{ is \emph{$\mathbf{d}$-$\mathcal{T}\!$-continuous}}\quad&\Leftrightarrow\quad\forall x\in X\ \exists\text{$\mathbf{d}$-Cauchy }(x_\lambda)\subseteq X\ (x_\lambda\xrightarrow{\mathcal{T}}x).
\end{align*}
\end{dfn}

We drop $\mathbf{d}$ when it is clear from the context.

Note these notions are trivial for hemimetric $\mathbf{d}$.  Indeed, if $x\mathbf{d}x=0$ then any constant $x$-valued net is $\mathbf{d}$-Cauchy, with limit $x$ for any topology $\mathcal{T}$.  It follows that $\mathbf{d}^\circ_\circ$-continuity, i.e. saying that each $x\in X$ is a $\mathbf{d}^\circ_\circ$-limit of a $\mathbf{d}$-Cauchy net, is trivial, as this forces $\mathbf{d}$ to be a hemimetric, by \cite[(8.16)]{Bice2019a}.  Likewise, if $\mathcal{R}$ is $\mathbf{d}$-$\sup$ or $\mathbf{d}$-$\max$ then $x\mathcal{R}\{x\}$ whenever $x\mathbf{d}x=0$.  Again it follows that $\mathbf{d}$-$\sup$-continuity, i.e. saying that each $x\in X$ is a $\mathbf{d}$-supremum of some $\mathbf{d}$-directed subset, is trivial, as this forces $\mathbf{d}$ to be a hemimetric, by \cite[(10.3)]{Bice2019a}.

Thus we primarily interested in $\mathbf{d}^\bullet_\circ$-continuity and $\mathbf{d}$-$\max$-continuity.  Indeed, classical domains can also be characterized by $\ll^\bullet_\circ$-continuity/$\ll$-$\max$-continuity rather than the more standard $\ll$-$\leq^\circ_\circ$-continuity/$\ll$-($\leq$-$\sup$)-continuity mentioned above for continuous posets (see \autoref{Tdomaineqs}/\autoref{Rdomaineqs} below).

First we wish to show how continuity can be characterised by certain interpolation conditions in $\mathbf{GRel}$.  The motivation here comes from the fact that the standard definition of abstract basis does not involve maxima, as we mentioned above, but rather the interpolation condition
\begin{equation}\label{ABInterpolation}
F\ll x\qquad\Rightarrow\qquad\exists y\in X\ (F\ll y\ll x),
\end{equation}
for all $F\in\mathcal{F}(X)$.  To generalise this, let us define $\mathcal{F}\mathbf{d}\in[0,\infty]^{\mathcal{F}(X)\times X}$ by
\[F(\mathcal{F}\mathbf{d})x=F\mathbf{d}x=\sup_{y\in F}y\mathbf{d}x.\]
So \eqref{ABInterpolation} can be expressed as $(\mathcal{F}\!\ll)\subseteq(\mathcal{F}\!\ll)\,\circ\ll$ (identifying relations with characteristic functions as usual).  This interpolation condition could be interpreted in various ways for more general distances.  The first condition that no doubt springs to mind is $\mathcal{F}\mathbf{d}\circ\mathbf{d}\leq\mathcal{F}\mathbf{d}$, but this is too weak to characterise either version of continuity we have defined.  The condition characterising $\mathbf{d}$-$\max$-continuity is rather $\mathcal{F}\mathbf{d}\ \circ\leq^\mathbf{d}\ \leq\mathcal{F}\mathbf{d}$, while $\mathbf{d}^\bullet_\circ$-continuity is characterised by the slightly weaker condition $\mathcal{F}\mathbf{d}\circ\Phi^\mathbf{d}\leq\mathcal{F}\mathbf{d}$, as we now proceed to show.

Recall that Smyth convergence can be characterised as follows.
\[x_\lambda\barrowc x\qquad\Leftrightarrow\qquad\lim_\lambda\mathbf{d}x_\lambda=\mathbf{d}x\qquad\Leftrightarrow\qquad\forall y\in X\ (y\mathbf{d}x_\lambda\rightarrow y\mathbf{d}x).\]

\begin{thm}\label{dbhcont}
The following are equivalent.
\begin{enumerate}
\item\label{dbhcont1} $X$ is $\mathbf{d}^\bullet_\circ$-continuous.
\item\label{dbhcont2} $\mathcal{F}\mathbf{d}\circ\Phi^\mathbf{d}\leq\mathcal{F}\mathbf{d}$.
\item\label{dbhcont2.5} $\mathcal{F}\mathbf{d}\circ\mathbf{d}\precapprox\mathcal{F}\mathbf{d}$ and $\mathbf{d}\circ\Phi^\mathbf{d}\leq\mathbf{d}$.
\item\label{dbhcont3} For any $\mathbf{\underline{d}}$-Cauchy $(x_\lambda)\subseteq X$, we have $\mathbf{d}$-Cauchy $(y_\gamma)\subseteq X$ with
\[(x_\lambda)\mathbf{\underline{d}}=(y_\gamma)\mathbf{\underline{d}}\qquad\text{and}\qquad\mathbf{d}(x_\lambda)=\mathbf{d}(y_\gamma).\]
\end{enumerate}
\end{thm}

\begin{proof}\
\begin{itemize}
\item[\eqref{dbhcont3}$\Rightarrow$\eqref{dbhcont1}]  Take $(x_\lambda)$ to be a constant net.

\item[\eqref{dbhcont1}$\Rightarrow$\eqref{dbhcont2}]  If $x_\lambda\barrowc x$ then, for any $y\in X$, we have $y\mathbf{d}x_\lambda\rightarrow y\mathbf{d}x$.  If $(x_\lambda)$ is also $\mathbf{d}$-Cauchy then $x_\lambda\mathbf{d}x\rightarrow0$.  Thus, for any $F\in\mathcal{F}(X)$ and $\epsilon>0$, we have some $x_\lambda$ with $F\mathbf{d}x_\lambda<F\mathbf{d}x+\epsilon$ and $x_\lambda\mathbf{d}x<\epsilon$, i.e. $\mathcal{F}\mathbf{d}\circ\Phi^\mathbf{d}\leq\mathcal{F}\mathbf{d}$.

\item[\eqref{dbhcont2}$\Rightarrow$\eqref{dbhcont2.5}]  Assuming \eqref{dbhcont2}, we immediately have $\mathbf{d}\circ\Phi^\mathbf{d}\leq\mathbf{d}$.  Also $\mathcal{F}\mathbf{d}\circ\mathbf{d}\precapprox\mathcal{F}\mathbf{d}$, as
\[\mathcal{F}\mathbf{d}\circ\mathbf{d}\leq\sup_{n\in\mathbb{N}}\mathcal{F}\mathbf{d}\circ n\mathbf{d}=\mathcal{F}\mathbf{d}\circ\Phi^\mathbf{d}\leq\mathcal{F}\mathbf{d}.\]

\item[\eqref{dbhcont2.5}$\Rightarrow$\eqref{dbhcont3}]  Assume \eqref{dbhcont2.5} and take $\epsilon>0$, $F\in\mathcal{F}(X)$ and $x\in X$.  We claim that we have $z<^\mathbf{d}_\epsilon x$ with $y\mathbf{d}z<y\mathbf{d}x+\epsilon$, for all $y\in F$.  Indeed, for each $y\in F$, we have $y'\in X$ such that $y\mathbf{d}y'\leq y\mathbf{d}x+\frac{1}{2}\epsilon$ and $\frac{\mathcal{F}\mathbf{d}\circ\mathbf{d}}{\mathcal{F}\mathbf{d}}(y'\mathbf{d}x)<\frac{1}{2}\epsilon$.  Thus $F'(\mathcal{F}\mathbf{d}\circ\mathbf{d})x\leq\frac{\mathcal{F}\mathbf{d}\circ\mathbf{d}}{\mathcal{F}\mathbf{d}}(F'\mathbf{d}x)<\frac{1}{2}\epsilon$, where $F'=\{y':y\in F\}$, i.e. we have $z\in X$ with $F'\mathbf{d}z+z\mathbf{d}x<\frac{1}{2}\epsilon$, so $F'\mathbf{d}z<\frac{1}{2}\epsilon$ and $z<^\mathbf{d}_\epsilon x$.  By \eqref{triangle}, $y\mathbf{d}z\leq y\mathbf{d}y'+y'\mathbf{d}z\leq y\mathbf{d}x+\epsilon$, for all $y\in F$.

Now take $\mathbf{\underline{d}}$-Cauchy $(x_\lambda)_{\lambda\in\Lambda}\subseteq X$ and consider $\Gamma=\mathcal{F}(X)\times\Lambda\times(0,\infty)$ directed by $\subseteq\times\prec\times>$.  By the claim, for every $(F,\lambda,\epsilon)\in\Gamma$, we have $y_{(F,\lambda,\epsilon)}\in X$ with $y\mathbf{d}y_{(F,\lambda,\epsilon)}<y\mathbf{d}x_\lambda+\epsilon$, for all $y\in F$, and $y_{(F,\lambda,\epsilon)}\mathbf{d}x_\lambda<\epsilon$.  This implies $\mathbf{d}(y_\gamma)\leq\mathbf{d}(x_\lambda)$ and $\mathbf{d}(x_\lambda)\leq\mathbf{d}(y_\gamma)$ respectively, as $\mathbf{d}$ is a distance.  Thus, by \cite[(7.3)]{Bice2019a},
\[(x_\lambda)\mathbf{\underline{d}}y=X(\mathbf{d}y-\mathbf{d}(x_\lambda))_+=X(\mathbf{d}y-\mathbf{d}(y_\gamma))_+=(y_\gamma)\mathbf{\underline{d}}y.\]
To see that $(y_\gamma)$ is $\mathbf{d}$-pre-Cauchy, note that
\begin{align*}
&\ \limsup_{(F,\lambda,\epsilon)\in\Gamma}\limsup_{(G,\beta,\delta)\in\Gamma}y_{(F,\lambda,\epsilon)}\mathbf{d}y_{(G,\beta,\delta)}\\
=&\ \limsup_{(F,\lambda,\epsilon)\in\Gamma}\limsup_{\substack{(G,\beta,\delta)\in\Gamma\\ y_{(F,\lambda,\epsilon)}\in G}}y_{(F,\lambda,\epsilon)}\mathbf{d}y_{(G,\beta,\delta)}\\
\leq&\ \limsup_{(F,\lambda,\epsilon)\in\Gamma}\limsup_{\substack{(G,\beta,\delta)\in\Gamma\\ y_{(F,\lambda,\epsilon)}\in G}}y_{(F,\lambda,\epsilon)}\mathbf{d}x_\beta+\delta\\
=&\ \limsup_{(F,\lambda,\epsilon)\in\Gamma}\limsup_{\beta\in\Lambda}y_{(F,\lambda,\epsilon)}\mathbf{d}x_\beta\\
\leq&\ \limsup_{(F,\lambda,\epsilon)\in\Gamma}\limsup_{\beta\in\Lambda}(y_{(F,\lambda,\epsilon)}\mathbf{d}x_\lambda+x_\lambda\underline{\mathbf{d}}x_\beta)\\
\leq&\ \limsup_{(F,\lambda,\epsilon)\in\Gamma}(\epsilon+\limsup_{\beta\in\Lambda}x_\lambda\underline{\mathbf{d}}x_\beta)\\
=&\ \limsup_{\lambda\in\Lambda}\limsup_{\beta\in\Lambda}x_\lambda\underline{\mathbf{d}}x_\beta\\
=&\ 0,\qquad\text{as $(x_\lambda)$ is $\mathbf{\underline{d}}$(-pre)-Cauchy.}
\end{align*}
Thus $(y_\gamma)$ has a $\mathbf{d}$-Cauchy subnet, by \cite[Theorem 7.3 (1)]{Bice2019a}.\qedhere
\end{itemize}
\end{proof}

Next we characterize $\mathbf{d}$-$\max$-continuity.

\begin{thm}\label{maxctsequiv}
The following are equivalent.
\begin{enumerate}
\item\label{maxcts} $X$ is $\mathbf{d}$-$\max$-continuous.
\item\label{Fd<d<Fd} $\mathcal{F}\mathbf{d}\circ\mathbin{\leq^\mathbf{d}}\leq\mathcal{F}\mathbf{d}$.
\item\label{Fd<d<approxFd} $\mathbf{d}\circ\mathbin{\leq^\mathbf{d}}\leq\mathbf{d}$ and $\leq^{\mathcal{F}\mathbf{d}}\ \subseteq\Phi^{\mathcal{F}\mathbf{d}}\circ\mathbf{\leq^\mathbf{d}}$.
\item\label{maxlimcts} $\mathbf{d}\circ\mathbin{\leq^\mathbf{d}}\precapprox\mathbf{d}$ and $X$ is $\mathbf{d}^\bullet_\circ$-continuous.
\item\label{dbardirected} For any $\mathbf{\underline{d}}$-directed $Y\subseteq X$, we have $\mathbf{d}$-directed $Z\subseteq X$ with
\[Y\underline{\mathbf{d}}=Z\underline{\mathbf{d}}\qquad\text{and}\qquad\mathbf{d}Y=\mathbf{d}Z.\]
\end{enumerate}
\end{thm}

\begin{proof}\
\begin{itemize}
\item[\eqref{dbardirected}$\Rightarrow$\eqref{maxcts}]  Take $Y=\{x\}$, for any $x\in X$.

\item[\eqref{maxcts}$\Rightarrow$\eqref{Fd<d<Fd}]  By \eqref{maxcts}, for any $x\in X$, we have $\mathbf{d}$-directed $Y\subseteq X$ with $x=\mathbf{d}$-$\max Y$.  By \cite[(9.1)]{Bice2019a}, for any $F\in\mathcal{F}(X)$, $F(\mathcal{F}\mathbf{d}\circ\mathbin{\leq^\mathbf{d}})x\leq(F\mathbf{d})Y=F(\mathbf{d}Y)=F\mathbf{d}x$.

\item[\eqref{Fd<d<Fd}$\Rightarrow$\eqref{Fd<d<approxFd}]  By \eqref{dbhcont2}, we immediately have $\mathbf{d}\circ\mathbin{\leq^\mathbf{d}}\leq\mathbf{d}$ and $\mathcal{F}\mathbf{d}\circ\mathbin{\leq^\mathbf{d}}\leq\ \leq^{\mathcal{F}\mathbf{d}}$ and hence $\Phi^{\mathcal{F}\mathbf{d}}\circ\mathbin{\leq^\mathbf{d}}=\sup_{n\in\mathbb{N}}n(\mathcal{F}\mathbf{d}\circ\mathbin{\leq^\mathbf{d}})\leq\ \leq^{\mathcal{F}\mathbf{d}}$, in other words $\leq^{\mathcal{F}\mathbf{d}}\ \subseteq\Phi^{\mathcal{F}\mathbf{d}}\circ\mathbf{\leq^\mathbf{d}}$.

\item[\eqref{Fd<d<Fd}$\Rightarrow$\eqref{maxlimcts}]  Note $\mathcal{F}\mathbf{d}\circ\mathbin{\leq^\mathbf{d}}\leq\mathcal{F}\mathbf{d}$ implies $\mathbf{d}\circ\mathbin{\leq^\mathbf{d}}\leq\mathbf{d}$ and hence $\mathbf{d}\circ\mathbin{\leq^\mathbf{d}}\precapprox\mathbf{d}$.   Also $\mathcal{F}\mathbf{d}\circ\mathbin{\leq^\mathbf{d}}\leq\mathcal{F}\mathbf{d}$ implies $\mathcal{F}\mathbf{d}\circ\Phi^\mathbf{d}\leq\mathcal{F}\mathbf{d}$, which means $X$ is $\mathbf{d}^\bullet_\circ$-continuous, by \autoref{dbhcont} \eqref{dbhcont2}.

\item[\eqref{maxlimcts}$\Rightarrow$\eqref{Fd<d<Fd}]  Assuming $\mathbf{d}\circ\mathbin{\leq^\mathbf{d}}\precapprox\mathbf{d}$ and \autoref{dbhcont} \eqref{dbhcont2}, for any $F\in\mathcal{F}(X)$, $x\in X$ and $\epsilon>0$, we have $z\in X$ with $F\mathbf{d}z\leq F\mathbf{d}x+\epsilon$ and $\frac{\mathbf{d}\circ\mathbin{\leq^\mathbf{d}}}{\mathbf{d}}(z\mathbf{d}x)<\epsilon$.  Thus we have $y\leq^\mathbf{d}x$ with $z\mathbf{d}y<\epsilon$ and hence $F\mathbf{d}y\leq F\mathbf{d}z+z\mathbf{d}y\leq F\mathbf{d}x+2\epsilon$.

\item[\eqref{Fd<d<approxFd}$\Rightarrow$\eqref{dbardirected}]  Assume \eqref{Fd<d<approxFd} and let $Z=\bigcup_{x\in Y}(\leq^\mathbf{d}x)$, so $\mathbf{d}Y\leq\mathbf{d}Z$.  Note that
\begin{align*}
\mathbf{d}\circ\mathbin{\leq^\mathbf{d}}\leq\mathbf{d}\qquad&\Leftrightarrow\qquad\text{$x=\mathbf{d}$-$\max\,(\leq^\mathbf{d}x)$, for all }x\in X.\\
\leq^{\mathcal{F}\mathbf{d}}\ \subseteq\Phi^{\mathcal{F}\mathbf{d}}\circ\mathbf{\leq^\mathbf{d}}\qquad&\Leftrightarrow\qquad(\leq^\mathbf{d}x)\text{ is $\mathbf{d}$-directed, for all }x\in X.
\end{align*}
So $\mathbf{d}Z=\inf_{x\in Y}\mathbf{d}(\leq^\mathbf{d}x)\leq\inf_{x\in Y}\mathbf{d}x=\mathbf{d}Y$.  Thus, as in the proof of \cite[(10.4)]{Bice2019a}, $Y\underline{\mathbf{d}}w=X(\mathbf{d}w-\mathbf{d}Y)_+=X(\mathbf{d}w-\mathbf{d}Z)_+=Z\underline{\mathbf{d}}w$.

For any $F\in\mathcal{F}(Z)$, we have $F'\in\mathcal{F}(Y)$ with $F\subseteq\bigcup_{x\in F'}(\leq^\mathbf{d}x)$.  Thus
\begin{align*}
(F\mathbf{d})Z&=\inf_{x\in Y}(F\mathbf{d})(\leq^\mathbf{d}x)=\inf_{x\in Y}F(\mathbf{d}(\leq^\mathbf{d}x))\leq\inf_{x\in Y}F\mathbf{d}x\leq\inf_{x\in Y}F(\mathbf{d}\circ\underline{\mathbf{d}})x\\
&\leq\inf_{x\in Y}\sup_{z\in F}\inf_{y\in F'}(z\mathbf{d}y+y\underline{\mathbf{d}}x)\leq\inf_{x\in Y}\sup_{z\in F}(z\mathbf{d}F'+F'\underline{\mathbf{d}}x)=(F'\underline{\mathbf{d}})Y=0,
\end{align*}
as $Y$ is $\underline{\mathbf{d}}$-directed, showing that $Z$ is $\mathbf{d}$-directed.\qedhere
\end{itemize}
\end{proof}

In particular, taking $\mathbf{d}=\ \ll$ in \autoref{dbhcont} and \autoref{maxctsequiv}, for some transitive relation $\ll\ \subseteq X\times X$, we see that our notions of continuity do indeed agree with the usual interpolation condition defining abstract bases, i.e.
\[X\text{ is $\ll^\bullet_\circ$-continuous}\quad\Leftrightarrow\quad X\text{ is $\ll$-$\max$-continuous}\quad\Leftrightarrow\quad\mathcal{F}\!\ll\ \,\subseteq\,\mathcal{F}\!\ll\circ\ll.\]

In \autoref{maxctsequiv}, we saw that $\mathbf{d}$-$\max$-continuity implies $\mathbf{d}^\bullet_\circ$-continuity.  Conversely, we can derive $\mathbf{d}$-$\max$-continuity (and slightly stronger continuity notions) from $\mathbf{d}^\bullet_\circ$-continuity under certain interpolation conditions, just like in \cite[Corollary 11.8]{Bice2019a}.  Below in \eqref{ctscor1}, $<^\mathbf{d}$ is the strict order defined in \cite[\S5]{Bice2019a} by
\begin{equation}\label{StrictOrder}
x<^\mathbf{d}y\qquad\Leftrightarrow\qquad(x\leq^\mathbf{d})\text{ is a $\underline{\mathbf{d}}^\bullet$-neighbourhood of }y
\end{equation}
and $\mathbf{d}\mathcal{P}\in[0,\infty]^{X\times\mathcal{P}(X)}$ is defined by $x(\mathcal{P}\mathbf{d})Y=x\mathbf{d}Y=\inf_{y\in Y}x\mathbf{d}y$.

\begin{cor}\label{ctscor} Assume $X$ is $\mathbf{d}^\bullet_\circ$-continuous.
\begin{enumerate}
\item\label{ctscor1} $X$ is $<^\mathbf{d}$-$\mathbf{d}$-$\max$-continuous if $\underline{\mathbf{d}}\circ\mathbin{\leq^{\mathbf{d}\mathcal{P}}}\precapprox\mathbf{d}\mathcal{P}$.

\item\label{ctscor2} $X$ is $(\mathbf{d}$-$)\mathbf{d}$-$\max$-continuous if $\mathbf{d}\circ\mathbin{\leq^\mathbf{d}}\precapprox\,\mathbf{d}$.

\item\label{ctscor3} $X$ is $(\mathbf{d}$-$)\mathbf{d}$-$\max$-continuous if $\mathbf{e}\circ\Phi^{\overline{\mathbf{d}}}\,\precapprox\,\mathbf{d}$, $\underline{\mathbf{d}}\vee\overline{\mathbf{d}}^\mathrm{op}\precapprox\mathbf{e}$, $X$ is $\mathbf{e}_\circ$-complete.

\item\label{ctscor4} $X$ is $\leq^\mathbf{d}$-$\mathbf{d}$-$\max$-continuous if $\mathbf{e}\circ\Phi^\mathbf{d}\,\precapprox\,\mathbf{d}$, $\underline{\mathbf{d}}\vee\overline{\mathbf{d}}^\mathrm{op}\precapprox\mathbf{e}$, $X$ is $\mathbf{e}_\circ$-complete\\
\null\hfill and $\overline{\mathbf{d}}^\bullet_\bullet$-separable.
\end{enumerate}
\end{cor}

\begin{proof}
Proving these results relies on using the interpolation conditions to define directed $Y\equiv^\mathbf{d}(x_\lambda)$ from $\mathbf{d}$-Cauchy $(x_\lambda)$.  Specifically, \eqref{ctscor1}, \eqref{ctscor3} and \eqref{ctscor4} follow from \cite{Bice2019a}[Theorems 11.3, 11.6 and 11.7] respectively, while \eqref{ctscor2} is just \autoref{maxctsequiv} \eqref{maxlimcts}, stated here again for comparison.
\end{proof}

If we require the $\mathbf{d}$-Cauchy or $\mathbf{d}$-directed subsets in \autoref{ctsdef} to lie in some subset $B$ of $X$, we get a generalised notion of a basis \textendash\, see \cite{Goubault2013} Definition 5.1.22.

\begin{dfn}\label{basisdef}
For any $B\subseteq X$, topology $\mathcal{T}$ on $X$ and $\mathcal{R}\subseteq X\times\mathcal{P}(X)$, define
\begin{align*}
B\text{ is a \emph{$\mathbf{d}$-$\mathcal{T}\!$-basis}}\quad&\Leftrightarrow\quad\forall x\in X\ \exists\text{$\mathbf{d}$-Cauchy }(x_\lambda)\subseteq B\ (x_\lambda\xrightarrow{\mathcal{T}}x).\\
B\text{ is a \emph{$\mathbf{d}$-$\mathcal{R}$-basis}}\quad&\Leftrightarrow\quad\forall x\in X\ \exists\text{$\mathbf{d}$-directed }Y\subseteq B\ (x\mathcal{R}Y).
\end{align*}
\end{dfn}

Bases can be characterised as in \autoref{dbhcont} and \autoref{maxctsequiv}, just with interpolation in $B$ rather than $X$.  If we already know that $X$ itself is continuous, then we can characterise bases with even weaker conditions.

Note $\mathbf{d}\circ B\circ\mathbf{d}$ below is like $\mathbf{d}\circ\mathbf{d}$, just with interpolation in $B$ instead of $X$, i.e.
\[x(\mathbf{d}\circ B\circ\mathbf{d})y=\inf_{b\in B}(x\mathbf{d}b+b\mathbf{d}y).\]
Equivalently, $\mathbf{d}\circ B\circ\mathbf{d}$ is the same as $\mathbf{d}\ \circ=_B\circ\ \mathbf{d}$, where $=_B$ is the relation $x=y\in B$ identified with its characteristic function.

\begin{prp}
If $X$ is $\mathbf{d}^\bullet_\circ$-continuous then, for any $B\subseteq X$,
\[B\text{ is a $\mathbf{d}^\bullet_\circ$-basis}\qquad\Leftrightarrow\qquad\mathbf{d}\circ B\circ\mathbf{d}\precapprox\mathbf{d}\qquad\Leftrightarrow\qquad B\text{ is $\mathbf{d}^\bullet_\bullet$-dense}.\]
\end{prp}

\begin{proof}
If $B$ is a is $\mathbf{d}^\bullet_\circ$-basis then, as in \autoref{dbhcont}, we have $\mathcal{F}\mathbf{d}\circ B\circ\Phi^\mathbf{d}\leq\mathcal{F}\mathbf{d}$ which certainly implies $\mathbf{d}\circ B\circ\mathbf{d}\precapprox\mathbf{d}$.  Conversely, if $\mathbf{d}\circ B\circ\mathbf{d}\precapprox\mathbf{d}$ then, for all $\epsilon>0$, we have $\delta>0$ such that $<^\mathbf{d}_\delta\ \subseteq\ <^\mathbf{d}_\epsilon\circ\,B\,\circ<^\mathbf{d}_\epsilon$ and hence 
\[\mathcal{F}\mathbf{d}\ \geq\ \mathcal{F}\mathbf{d}\circ\Phi^\mathbf{d}\ \geq\ \mathcal{F}\mathbf{d}\ \circ<^\mathbf{d}_\delta\ \geq\ \mathcal{F}\mathbf{d}\ \circ<^\mathbf{d}_\epsilon\circ\ B\ \circ<^\mathbf{d}_\epsilon.\]
As $\epsilon>0$ was arbitrary, $\mathcal{F}\mathbf{d}\geq\mathcal{F}\mathbf{d}\circ B\circ\Phi^\mathbf{d}$ so $B$ is a $\mathbf{d}^\bullet_\circ$-basis.

Assume again that $B$ is a $\mathbf{d}^\bullet_\circ$-basis and take non-empty $\mathbf{d}^\bullet_\bullet$-open $O\subseteq X$.  So for any $x\in O$, we have $y_1,\ldots,y_m,z_1,\ldots,z_n\in X$ and $r_1,\ldots,r_m,s_1,\ldots,s_n>0$ with
\[x\in N=\bigcap_{1\leq j\leq m}(y_j)_{r_j}^\bullet\cap\bigcap_{1\leq k\leq n}(z_k)_\bullet^{s_k}\subseteq O.\]
As in the proof of \autoref{dbhcont} \eqref{dbhcont2.5}$\Rightarrow$\eqref{dbhcont3} we then have $b\in B$ with $y_j\mathbf{d}b<r_j$, for $1\leq j\leq m$, and $b\mathbf{d}x<\min\limits_{1\leq k\leq n}s_k-x\mathbf{d}z_k$.  Thus $b\in N\subseteq O$, i.e. $B$ is $\mathbf{d}^\bullet_\bullet$-dense.

Conversely, if $B$ is $\mathbf{d}^\bullet_\bullet$-dense then $\mathbf{d}\circ B\circ\mathbf{d}=\mathbf{d}\circ\mathbf{d}\leq\mathbf{d}$.
\end{proof}

Let $\mathbf{d}_B$ denote the restriction of $\mathbf{d}$ to $B\times B$.

\begin{prp}\label{reflexbasis}
If $B\subseteq X$ is a $\mathbf{d}^\bullet_\circ$-basis then $\overline{\mathbf{d}|_B}=\overline{\mathbf{d}}|_B$ and $\underline{\mathbf{d}|_B}=\underline{\mathbf{d}}|_B$.
\end{prp}

\begin{proof}
As noted above, if $B\subseteq X$ is a $\mathbf{d}^\bullet_\circ$-basis then $\mathbf{d}\circ B\circ\mathbf{d}\leq\mathbf{d}$.  Thus $\mathbf{d}\circ B\circ\underline{\mathbf{d}}\leq\mathbf{d}$ and $\overline{\mathbf{d}}\circ B\circ\mathbf{d}\leq\mathbf{d}$ so \cite[Proposition 2.5]{Bice2019a} yields $\overline{\mathbf{d}|_B}=\overline{\mathbf{d}}|_B$ and $\underline{\mathbf{d}|_B}=\underline{\mathbf{d}}|_B$.
\end{proof}

If we join the $\mathbf{d}^\bullet$-topology with $(\leq^\mathbf{d})_\bullet$ instead of $\mathbf{d}_\bullet$, we get an analogous characterization of $\mathbf{d}$-$\max$-bases.  We omit the proof, which is much the same as above.

\begin{prp}
If $X$ is $\mathbf{d}$-$\max$-continuous then, for any $B\subseteq X$,
\[B\text{ is a $\mathbf{d}$-$\max$-basis}\quad\Leftrightarrow\quad\mathbf{d}\,\circ\,B\ \circ\leq^\mathbf{d}\ \precapprox\,\mathbf{d}\quad\Leftrightarrow\quad B\text{ is $(\mathbf{d}^\bullet\vee(\leq^\mathbf{d})_\bullet)$-dense}.\]
\end{prp}

\section{Way-Below Distances}\label{WBD}

Next we consider distance analogs of the way-below relation.

\begin{dfn}
For any topology $\mathcal{T}$ on $X$ and relation $\mathcal{R}\subseteq X\times\mathcal{P}(X)$, define
\begin{align}
\label{Tway}\mathcal{T}\mathbf{d}(x,y)&=\sup\{(x\mathbf{d}(z_\lambda)-y\mathbf{d}z)_+:(z_\lambda)\text{ is $\mathbf{d}$-Cauchy and }z_\lambda\xrightarrow{\mathcal{T}}z\}.\\
\label{Rway}\mathcal{R}\mathbf{d}(x,y)&=\sup\{(x\mathbf{d}Z-y\mathbf{d}z)_+:Z\text{ is $\mathbf{d}$-directed and $z\mathcal{R}Z$}\}.
\end{align}
\end{dfn}

Again, we abbreviate duplicate distance symbols, e.g. $\mathbf{d}^\circ_\circ\mathbf{d}$ and $\mathbf{d}$-$\sup\mathbf{d}$ are written as $^\circ_\circ\mathbf{d}$ and $\sup\mathbf{d}$, which are the cases of primary interest.  Indeed, $^\bullet_\circ\mathbf{d}$ and $\max\mathbf{d}$ coincide with $\overline{\mathbf{d}}$, as long as $X$ is continuous w.r.t. $\mathbf{d}^\bullet_\circ$ and $\mathbf{d}$-$\max$ respectively.

Way-below distances are essentially inverse to upper/lower hemimetrics, as we will see very shortly.  The first thing to note is that, while $\overline{\mathbf{d}}$ and $\underline{\mathbf{d}}$ turn a general distance $\mathbf{d}$ into a weaker hemimetric, $\mathcal{T}\mathbf{d}$ and $\mathcal{R}\mathbf{d}$ instead turn a hemimetric $\mathbf{d}$ into a stronger (usually non-hemimetric) distance.  First we consider $\mathcal{T}\mathbf{d}$.

\begin{prp}\label{WBprops}
If $\mathbf{d}$ is a hemimetric then $\mathcal{T}\mathbf{d}$ is a distance and, moreover,
\[\underline{\mathcal{T}\mathbf{d}}\vee\overline{\mathcal{T}\mathbf{d}}\leq\mathbf{d}\leq\mathcal{T}\mathbf{d}.\]
\end{prp}

\begin{proof}
Taking $(z_\lambda)$ and $z$ to be $y$ in \eqref{Tway} shows that $\mathbf{d}\leq\mathcal{T}\mathbf{d}$.

As $\mathbf{d}$ is a distance, for all $w,x,y,z\in X$ and $(z_\lambda)\subseteq X$,
\begin{align*}
w\mathbf{d}z&\leq w\mathbf{d}y+y\mathbf{d}z\quad\text{so}\\
x\mathbf{d}z_\lambda-y\mathbf{d}z&\leq x\mathbf{d}z_\lambda-w\mathbf{d}z+w\mathbf{d}y\quad\text{and hence}\\
x\mathcal{T}\mathbf{d}y&\leq x\mathcal{T}\mathbf{d}w+w\mathbf{d}y.\\
\intertext{So $\mathcal{T}\mathbf{d}\leq\mathcal{T}\mathbf{d}\circ\mathbf{d}$, i.e. $\underline{\mathcal{T}\mathbf{d}}\leq\mathbf{d}$.  Likewise,}
x\mathbf{d}z_\lambda&\leq x\mathbf{d}w+w\mathbf{d}z_\lambda\quad\text{so}\\
x\mathbf{d}z_\lambda-y\mathbf{d}z&\leq x\mathbf{d}w+w\mathbf{d}z_\lambda-y\mathbf{d}z\quad\text{and hence}\\
x\mathcal{T}\mathbf{d}y&\leq x\mathbf{d}w+w\mathcal{T}\mathbf{d}y.
\end{align*}
So $\mathcal{T}\mathbf{d}\leq\mathbf{d}\circ\mathcal{T}\mathbf{d}$, i.e. $\overline{\mathcal{T}\mathbf{d}}\leq\mathbf{d}$.

It follows that $\underline{\mathcal{T}\mathbf{d}}\leq\mathcal{T}\mathbf{d}$ and $\overline{\mathcal{T}\mathbf{d}}\leq\mathcal{T}\mathbf{d}$, either one of which is equivalent to saying that $\mathcal{T}\mathbf{d}$ is a distance (see \cite[(2.3)]{Bice2019a}).
\end{proof}

\begin{rmk}
Even if $\mathbf{d}$ is not a hemimetric or even a distance, we can still prove that $\mathcal{T}\mathbf{d}$ is a distance as long as the topology is at least as fine as the lower topology, i.e. $\mathbf{d}_\circ\subseteq\mathcal{T}$.  Then $\mathbf{d}z\leq\mathbf{d}(z_\lambda)$ whenever $z_\lambda\xrightarrow{\mathcal{T}}z$ and hence
\begin{align*}
(x\mathbf{d}(z_\lambda)-y\mathbf{d}z)_+&\leq(x\mathbf{d}(z_\lambda)-w\mathbf{d}z)_++(w\mathbf{d}z-y\mathbf{d}z)_+\\
&\leq(x\mathbf{d}(z_\lambda)-w\mathbf{d}z)_++(w\mathbf{d}(z_\lambda)-y\mathbf{d}z)_+\quad\text{so}\\
x\mathcal{T}\mathbf{d}y&\leq x\mathcal{T}\mathbf{d}w+w\mathcal{T}\mathbf{d}y.
\end{align*}
\end{rmk}

Order theory is consistently biased towards preorders over non-reflexive transitive relations, and domain theory is no exception.  Indeed, an unbiased definition would say a domain is not a poset but rather a set together with two relations, $\leq$ and $\ll$, each definable from the other, satisfying certain completeness and continuity conditions, which can again be stated equivalently in terms of $\leq$ or $\ll$.  This duality extends to quantitative domains, as the following result shows.

Recall our standing assumption that $\mathbf{e}$ is a distance, just like $\mathbf{d}$.

\begin{thm}\label{Tdomaineqs}
The following are equivalent.
\begin{enumerate}
\item\label{eTdomain} $X$ is $\mathbf{e}^\circ_\circ$-complete, $\mathbf{d}$-$\mathbf{e}^\circ_\circ$-continuous and $\mathbf{d}={}^\circ_\circ\mathbf{e}$.
\item\label{dTdomain} $X$ is $\mathbf{d}^\bullet_\circ$-complete, $\mathbf{d}^\bullet_\circ$-continuous and $\mathbf{e}=\underline{\mathbf{d}}\geq\overline{\mathbf{d}}$.
\end{enumerate}
\end{thm}

\begin{proof}\
\begin{itemize}
\item[\eqref{eTdomain}$\Rightarrow$\eqref{dTdomain}]  As $X$ is $\mathbf{d}$-$\mathbf{e}^\circ_\circ$-continuous, $\mathbf{e}$ is must be a hemimetric, by \cite[(8.16)]{Bice2019a}.  By \autoref{WBprops}, $\mathbf{d}={}^\circ_\circ\mathbf{e}$ is a distance with $\underline{\mathbf{d}}\vee\overline{\mathbf{d}}\leq\mathbf{e}\leq\mathbf{d}$.  As $X$ is $\mathbf{d}$-$\mathbf{e}^\circ_\circ$-continuous, if $x\in X$ we have $\mathbf{d}$-Cauchy $(x_\lambda)$ with $x=\mathbf{e}^\circ_\circ$-$\lim x_\lambda$.  By \cite[(7.4)]{Bice2019a}, $(x_\lambda)\mathbf{d}=(x_\lambda)\underline{\mathbf{d}}$ and, by \cite[(8.15)]{Bice2019a}, $x_\lambda\underline{\mathbf{d}}x\leq x_\lambda\mathbf{e}x\rightarrow0$ so
\[x\mathbf{e}y\leq(x_\lambda)\mathbf{e}y\leq(x_\lambda)\mathbf{d}y=(x_\lambda)\underline{\mathbf{d}}y\leq(x_\lambda)\underline{\mathbf{d}}x+x\underline{\mathbf{d}}y=x\underline{\mathbf{d}}y\leq x\mathbf{e}y,\]
i.e. $\mathbf{e}=\underline{\mathbf{d}}\geq\overline{\mathbf{d}}$.  Next we claim any $\mathbf{e}^\circ_\circ$-limit $x$ of $\mathbf{d}$-Cauchy $(x_\lambda)$ is a $\mathbf{d}^\bullet_\circ$-limit.  As above, $(x_\lambda)\mathbf{d}x=(x_\lambda)\mathbf{e}x=0$, so $x$ is a $\mathbf{d}_\circ$-limit, by \cite[(8.15)]{Bice2019a}.  By \cite[(7.4)]{Bice2019a}, $\mathbf{d}(x_\lambda)=\overline{\mathbf{d}}(x_\lambda)$ so
\[y\mathbf{d}(x_\lambda)=y\overline{\mathbf{d}}(x_\lambda)\leq y\mathbf{e}(x_\lambda)=(y\mathbf{e}(x_\lambda)-x\mathbf{e}x)_+\leq y\mathbf{d}x,\]
as $\mathbf{d}={}^\circ_\circ\mathbf{e}$, i.e. $\mathbf{d}(x_\lambda)\leq\mathbf{d}x$ so $x$ is also a $\mathbf{d}^\bullet$-limit.  Thus the claim is proved and hence $X$ is also $\mathbf{d}^\bullet_\circ$-continuous and $\mathbf{d}^\bullet_\circ$-complete.

\item[\eqref{dTdomain}$\Rightarrow$\eqref{eTdomain}]  Again $\mathbf{e}=\underline{\mathbf{d}}$ is a hemimetric.  As $X$ is $\mathbf{d}^\bullet_\circ$-continuous, for any $y\in X$, we have $\mathbf{d}$-Cauchy $y_\lambda\barrowc y$.  By \cite[(8.13)]{Bice2019a}, $(y_\lambda)\mathbf{e}=y\mathbf{e}$.  By \cite[(8.11) and (8.16)]{Bice2019a}, $y=\mathbf{e}^\circ_\circ$-$\lim y_\lambda$ so $X$ is $\mathbf{d}$-$\mathbf{e}^\circ_\circ$-continuous.  By \cite[(7.4)]{Bice2019a} again,
\[x{}^\circ_\circ\mathbf{e}y\geq(x\mathbf{e}(y_\lambda)-y\mathbf{e}y)_+\geq x\overline{\mathbf{d}}(y_\lambda)=x\mathbf{d}(y_\lambda)=x\mathbf{d}y,\]
i.e. ${}^\circ_\circ\mathbf{e}\geq\mathbf{d}$.  Now take $\mathbf{e}$-Cauchy $(z_\lambda)$.  By \autoref{dbhcont} \eqref{dbhcont3}, we have $\mathbf{d}$-Cauchy $(z'_\gamma)\subseteq X$ with $(z_\lambda)\mathbf{e}=(z'_\gamma)\mathbf{e}$ and $\mathbf{d}(z_\lambda)=\mathbf{d}(z'_\gamma)$.  As $X$ is $\mathbf{d}^\bullet_\circ$-complete, $(z'_\gamma)$ has a $\mathbf{d}^\bullet_\circ$-limit $z$.  By \cite[(8.11) and (8.16)]{Bice2019a} again, $z$ is also an $\mathbf{e}^\circ_\circ$-limit of $(z'_\gamma)$.  Thus $z\mathbf{e}=(z'_\gamma)\mathbf{e}=(z_\lambda)\mathbf{e}$, so $z$ is also an $\mathbf{e}^\circ_\circ$-limit of $(z_\lambda)$, i.e. $X$ is $\mathbf{e}^\circ_\circ$-complete.  On the other hand, if we are already given $z=\mathbf{e}^\circ_\circ$-$\lim z_\lambda$ and hence $z=\mathbf{e}^\circ_\circ$-$\lim z'_\gamma$ then $z'_\gamma\barrowc z$, by \cite[(8.14)]{Bice2019a} and $\mathbf{d}^\bullet_\circ$-completeness again.  Thus $\mathbf{e}(z_\lambda)\leq\mathbf{d}(z_\lambda)=\mathbf{d}(z'_\gamma)=\mathbf{d}z$ so, as $\mathbf{d}\leq\mathbf{d}\circ\mathbf{e}$,
\[(x\mathbf{e}(z_\lambda)-y\mathbf{e}z)_+\leq(x\mathbf{d}z-y\mathbf{e}z)_+\leq x\mathbf{d}y.\]
As $(z_\lambda)$ was arbitrary, $\mathbf{e}^\circ_\circ\leq\mathbf{d}$ and hence $\mathbf{d}=\mathbf{e}^\circ_\circ$.\qedhere
\end{itemize}
\end{proof}

We also have the following analogous results for the relational rather than topological notions, whose proofs are also very similar.

\begin{prp}\label{Rdprops}
If $\mathbf{d}$ is a hemimetric and $x\mathcal{R}\{x\}$, for all $x\in X$, then
\[\underline{\mathcal{R}\mathbf{d}}\vee\overline{\mathcal{R}\mathbf{d}}\leq\mathbf{d}\leq\mathcal{R}\mathbf{d}.\]
In particular, $\mathcal{R}\mathbf{d}$ is a a distance.
\end{prp}

\begin{proof}
Taking $Z=\{y\}$ and $z=y$ in \eqref{Rway} yields $x\mathbf{d}y\leq x\mathcal{R}\mathbf{d}y$.

As $\mathbf{d}$ is a distance,
\begin{align*}
w\mathbf{d}z&\leq w\mathbf{d}y+y\mathbf{d}z\quad\text{so}\\
x\mathbf{d}Z-y\mathbf{d}z&\leq x\mathbf{d}Z-w\mathbf{d}z+w\mathbf{d}y\quad\text{and hence}\\
x\mathcal{R}\mathbf{d}y&\leq x\mathcal{R}\mathbf{d}w+w\mathbf{d}y.\\
\intertext{So $\mathcal{R}\mathbf{d}\leq\mathcal{R}\mathbf{d}\circ\mathbf{d}$, i.e. $\underline{\mathcal{R}\mathbf{d}}\leq\mathbf{d}$.  Likewise,}
x\mathbf{d}Z&\leq x\mathbf{d}w+w\mathbf{d}Z\quad\text{so}\\
x\mathbf{d}Z-y\mathbf{d}z&\leq x\mathbf{d}w+w\mathbf{d}Z-y\mathbf{d}z\quad\text{and hence}\\
x\mathcal{R}\mathbf{d}y&\leq x\mathbf{d}w+w\mathcal{R}\mathbf{d}y.
\end{align*}
So $\mathcal{R}\mathbf{d}\leq\mathbf{d}\circ\mathcal{R}\mathbf{d}$, i.e. $\overline{\mathcal{R}\mathbf{d}}\leq\mathbf{d}$.
\end{proof}

\begin{rmk}
As before, even if $\mathbf{d}$ is not a hemimetric or even a distance, we can still prove that $\mathcal{R}\mathbf{d}$ is a distance as long as $\mathbf{d}z\leq\mathbf{d}Z$ whenever $z\mathcal{R}Z$ because then
\begin{align*}
(x\mathbf{d}Z-y\mathbf{d}z)_+&\leq(x\mathbf{d}Z-w\mathbf{d}z)_++(w\mathbf{d}z-y\mathbf{d}z)_+\\
&\leq(x\mathbf{d}Z-w\mathbf{d}z)_++(w\mathbf{d}Z-y\mathbf{d}z)_+\quad\text{so}\\
x\mathcal{R}\mathbf{d}y&\leq x\mathcal{R}\mathbf{d}w+w\mathcal{R}\mathbf{d}y.\qedhere
\end{align*}
\end{rmk}

\begin{thm}\label{Rdomaineqs}
The following are equivalent.
\begin{enumerate}
\item\label{eRdomain} $X$ is $\mathbf{e}$-$\sup$-complete, $\mathbf{d}$-$\mathbf{e}$-$\sup$-continuous and $\mathbf{d}=\sup\mathbf{e}$.
\item\label{dRdomain} $X$ is $\mathbf{d}$-$\max$-complete, $\mathbf{d}$-$\max$-continuous and $\mathbf{e}=\underline{\mathbf{d}}\geq\overline{\mathbf{d}}$.
\end{enumerate}
\end{thm}

\begin{proof}\
\begin{itemize}
\item[\eqref{eRdomain}$\Rightarrow$\eqref{dRdomain}]  As $X$ is $\mathbf{d}$-$\mathbf{e}$-$\sup$-continuous, $\mathbf{e}$ is a hemimetric, by \cite[(10.3)]{Bice2019a}.  By \autoref{Rdprops}, $\mathbf{d}=\sup\mathbf{e}$ is a distance with $\underline{\mathbf{d}}\vee\overline{\mathbf{d}}\leq\mathbf{e}\leq\mathbf{d}$.  As $X$ is $\mathbf{d}$-$\mathbf{e}$-$\sup$-continuous, if $x\in X$ we have $\mathbf{d}$-directed $Y$ with $x=\mathbf{e}$-$\sup Y$.  Thus $Y\underline{\mathbf{d}}x\leq Y\mathbf{e}x=0$ and, by \cite[(9.2)]{Bice2019a}, $Y\mathbf{d}=Y\underline{\mathbf{d}}$ so
\[x\mathbf{e}y\leq Y\mathbf{e}y\leq Y\mathbf{d}y=Y\underline{\mathbf{d}}y\leq Y\underline{\mathbf{d}}x+x\underline{\mathbf{d}}y=x\underline{\mathbf{d}}y\leq x\mathbf{e}y,\]
i.e. $\mathbf{e}=\underline{\mathbf{d}}\geq\overline{\mathbf{d}}$.  Next we claim that any $\mathbf{e}$-supremum $x$ of $\mathbf{d}$-directed $Y$ is a $\mathbf{d}$-maximum.  By \cite[(9.2)]{Bice2019a}, $Y\mathbf{d}x=Y\mathbf{e}x=0$, i.e. $Y\leq^\mathbf{d}x$.  Again by \cite[(9.2)]{Bice2019a}, and the fact $\mathbf{d}=\sup\mathbf{e}$,
\[y\mathbf{d}Y=y\overline{\mathbf{d}}Y\leq y\mathbf{e}Y=(y\mathbf{e}Y-x\mathbf{e}x)_+\leq y\mathbf{d}x,\]
i.e. $\mathbf{d}Y\leq\mathbf{d}x$ so $x$ is also a $\mathbf{d}$-maximum.  Thus the claim is proved and hence $X$ is also $\mathbf{d}$-$\max$-continuous and $\mathbf{d}$-$\max$-complete.

\item[\eqref{dRdomain}$\Rightarrow$\eqref{eRdomain}]  Again $\mathbf{e}=\underline{\mathbf{d}}$ is a hemimetric.  As $X$ is $\mathbf{d}$-$\max$-continuous, for any $y\in X$, we have $y=\mathbf{d}$-$\max Y$, for some $\mathbf{d}$-directed $Y$.  By \cite[(10.4)]{Bice2019a}, $y=\mathbf{e}$-$\sup Y$ so $X$ is $\mathbf{d}$-$\mathbf{e}$-$\sup$-continuous.  By \cite[(9.2)]{Bice2019a} again, $\overline{\mathbf{d}}Y=\mathbf{d}Y$ so
\[x(\sup\mathbf{e})y\geq(x\mathbf{e}Y-y\mathbf{e}y)_+\geq x\overline{\mathbf{d}}Y=x\mathbf{d}Y=x\mathbf{d}y,\]
i.e. $\sup\mathbf{e}\geq\mathbf{d}$.  Now take $\mathbf{e}$-directed $Z$.  By \autoref{maxctsequiv} \eqref{dbardirected}, we have $\mathbf{d}$-directed $Z'\subseteq X$ with $Z\mathbf{e}=Z'\mathbf{e}$ and $\mathbf{d}Z=\mathbf{d}Z'$.  As $X$ is $\mathbf{d}$-$\max$-complete, $Z'$ has a $\mathbf{d}$-maximum $z$.  By \cite[(10.4)]{Bice2019a}, $z$ is also an $\mathbf{e}$-supremum of $Z'$.  Thus $z\mathbf{e}=Z'\mathbf{e}=Z\mathbf{e}$, so $z$ is also an $\mathbf{e}$-supremum of $Z$, i.e. $X$ is $\mathbf{e}$-$\sup$-complete.  On the other hand, if we are already given $z=\mathbf{e}$-$\sup Z$ then $z=\mathbf{e}$-$\sup Z'$ so $z=\mathbf{d}$-$\max Z'$, by \cite[(10.6)]{Bice2019a} and $\mathbf{d}$-$\max$-completeness again.  Thus $\mathbf{e}Z\leq\mathbf{d}Z=\mathbf{d}Z'=\mathbf{d}z$ so, as $\mathbf{d}\leq\mathbf{d}\circ\mathbf{e}$,
\[(x\mathbf{e}Z-y\mathbf{e}z)_+\leq(x\mathbf{d}z-y\mathbf{e}z)_+\leq x\mathbf{d}y.\]
As $Z$ was arbitrary, $\sup\mathbf{e}\leq\mathbf{d}$ and hence $\mathbf{d}=\sup\mathbf{e}$.\qedhere
\end{itemize}
\end{proof}

We base our definition of domains on \autoref{Tdomaineqs} \eqref{dTdomain} and \autoref{Rdomaineqs} \eqref{dRdomain}.  This is dual to the usual focus on hemimetrics and preorders.

\begin{dfn}\label{domdefs}
For any topology $\mathcal{S}$ on $X$ or relation $\mathcal{S}\subseteq X\times\mathcal{P}(X)$, define
\begin{align*}
&X\text{ is a $\mathbf{d}$-$\mathcal{S}$-\emph{predomain}}&&\Leftrightarrow&&X\text{ is $\mathbf{d}$-$\mathcal{S}$-continuous and }\overline{\mathbf{d}}\leq\underline{\mathbf{d}}.\\
&X\text{ is a $\mathbf{d}$-$\mathcal{S}$\emph{-domain}}&&\Leftrightarrow&&X\text{ is a $\mathbf{d}$-$\mathcal{S}$-complete $\mathbf{d}$-$\mathcal{S}$-predomain}.
\end{align*}
\end{dfn}

For a poset $(\mathbb{P},\leq)$, with way-below relation $\ll\ ={}^\circ_\circ\!\!\!\leq\ =\sup\!\!\leq$, \autoref{domdefs} generalizes the notion of domain from \cite[Definition I-1.6]{GierzHofmannKeimelLawsonMisloveScott2003}.  Specifically
\begin{align*}
\text{$(\mathbb{P},\leq)$ is a domain}\qquad&\Leftrightarrow\qquad\mathbb{P}\text{ is a $\ll_\circ^\bullet$-domain with}\leq\ =\underline{\ll}\\
&\Leftrightarrow\qquad\mathbb{P}\text{ is a $\ll$-$\max$-domain with}\leq\ =\underline{\ll}.\\
\intertext{\autoref{domdefs} also generalizes `stratified predomain' from \cite[\S2.3]{Keimel2016}, i.e.}
\text{$(\mathbb{P},\prec\!\!\prec)$ is a stratified predomain}\qquad&\Leftrightarrow\qquad\mathbb{P}\text{ is a $\prec\!\!\prec_\circ^\bullet$-predomain}\\
&\Leftrightarrow\qquad\mathbb{P}\text{ is a $\prec\!\!\prec$-$\max$-predomain}.
\end{align*}
(on its own `predomain' in \cite[\S2.1]{Keimel2016} is synonymous with `abstract basis' and hence with $\prec\!\!\prec^\bullet_\circ$-continuity or $\prec\!\!\prec$-$\max$-continuity in our terminology).  While domains usually refer to posets rather than prosets, we are not requiring $\leq^{\underline{\mathbf{d}}}$ to be antisymmetric in \autoref{domdefs}.  Although we can always make $\leq^{\underline{\mathbf{d}}}$ antisymmetric, if so desired, by identifying $\mathbf{d}$-equivalent points (i.e. pairs $x,y\in X$ with $x\mathbf{d}=y\mathbf{d}$ and $\mathbf{d}x=\mathbf{d}y$), as $\overline{\mathbf{d}}\leq\underline{\mathbf{d}}$ implies that $x$ and $y$ are $\mathbf{d}$-equivalent iff $x\underline{\mathbf{d}}^\vee\!y=0$.

Under interpolation conditions like in \autoref{ctscor}, we can show that $\mathbf{d}^\bullet_\circ$-domains are just $\mathbf{d}$-$\max$-domains that are also complete in the usual sense with respect to the pseudometric $\underline{\mathbf{d}}^\vee$ (as in \cite{Bice2019a}, we denote the symmetrisation of $\mathbf{d}$ by $\mathbf{d}^\vee=\mathbf{d}\vee\mathbf{d}^\mathrm{op}$, noting completeness w.r.t. $\mathbf{d}^\vee_\circ=(\mathbf{d}^\vee)_\circ$ can be characterised in the usual way familiar from (pseudo)metric space theory, i.e. for every $\mathbf{d}^\vee$-Cauchy net $(x_\lambda)$, we have $x\in X$ with $x\mathbf{d}^\vee x_\lambda\rightarrow0$ \textendash\, see \cite[(8.15)]{Bice2019a}).

\begin{cor}\label{domcor}
\begin{gather*}
\text{If}\qquad\underline{\mathbf{d}}\circ\mathbin{\leq^{\mathbf{d}\mathcal{P}}}\precapprox\mathbf{d}\mathcal{P},\qquad\underline{\mathbf{d}}^\vee\circ\Phi^{\underline{\mathbf{d}}}\precapprox\mathbf{d}\qquad\text{or}\qquad\mathbf{d}\circ\mathbin{\leq^\mathbf{d}}\precapprox\mathbf{d}\text{ and }\leq^{\mathcal{F}\mathbf{d}}\circ\,\overline{\mathbf{d}}\leq\mathcal{F}\mathbf{d}\\
\text{then}\qquad\qquad X\text{ is a $\mathbf{d}^\bullet_\circ$-domain}\qquad\Leftrightarrow\qquad X\text{ is a $\underline{\mathbf{d}}^\vee_\circ$-complete $\mathbf{d}$-$\max$-domain}.
\end{gather*}
\end{cor}

\begin{proof}
If $X$ is a $\mathbf{d}^\bullet_\circ$-domain then $X$ is $\underline{\mathbf{d}}^\circ_\circ$-complete, by \autoref{Tdomaineqs}.  So any $\underline{\mathbf{d}}^\vee$-Cauchy $(x_n)$ has a $\underline{\mathbf{d}}^\circ_\circ$-limit, which is a $\underline{\mathbf{d}}^\vee_\circ$-limit, by \cite[(8.15)]{Bice2019a}, i.e. $X$ is $\underline{\mathbf{d}}^\vee_\circ$-complete.  As $X$ is $\mathbf{d}^\bullet_\circ$-continuous, any of the given interpolation conditions implies that $X$ is $\mathbf{d}$-$\max$-continuous, by \autoref{ctscor} (with $\mathbf{e}=\underline{\mathbf{d}}^\vee$ in the last case).  By \cite[(11.2)]{Bice2019a}, $X$ is also $\mathbf{d}$-$\max$-complete and hence a $\mathbf{d}$-$\max$-domain.

Conversely, say $X$ is a $\underline{\mathbf{d}}^\vee_\circ$-complete $\mathbf{d}$-$\max$-domain.  As $X$ is $\mathbf{d}$-$\max$-continuous, any of the given interpolation conditions then implies that $X$ is $\mathbf{d}^\bullet_\circ$-complete, by \cite[Corollary 11.8]{Bice2019a}.  By \autoref{maxctsequiv} \eqref{maxlimcts}, $X$ is also $\mathbf{d}^\bullet_\circ$-continuous and hence a $\mathbf{d}^\bullet_\circ$-domain.
\end{proof}

\section{Hausdorff Distances}\label{HD}

To complete predomains to domains, we need to find some larger space to embed them in.  Hyperspaces of subsets $\mathcal{P}(X)$ are a natural candidate, the only question is how to extend the distance from $X$ to $\mathcal{P}(X)$.

\begin{dfn}
For any $\mathbf{d}\in[0,\infty]^{X\times X}$, define $\mathbf{d}^\mathcal{H}$ and $\mathbf{d}_\mathcal{H}$ on $\mathcal{P}(X)$ by
\begin{align*}
Y\mathbf{d}^\mathcal{H}Z&=(Y\mathbf{d})Z=\inf_{z\in Z}\sup_{y\in Y}y\mathbf{d}z.\\
Y\mathbf{d}_\mathcal{H}Z&=Y(\mathbf{d}Z)=\sup_{y\in Y}\inf_{z\in Z}y\mathbf{d}z.
\end{align*}
\end{dfn}

The classical Hausdorff distance $\mathbf{d}_\mathcal{H}$ is well-known \textendash\, see \cite[Lemma 7.5.1]{Goubault2013} \textendash\, but the `reverse Hausdorff distance' $\mathbf{d}^\mathcal{H}$ does not appear to have been considered before.  This could be due to the focus on hemimetrics over general distances, as $\leq^{\mathbf{d}^\mathcal{H}}$ often fails to be reflexive, e.g. when $\mathbf{d}$ is a metric and $X$ has at least 2 points.  However, it is $\mathbf{d}^\mathcal{H}$ that we need to complete predomains to domains.

First we note some basic functorial properties.  In particular, it follows from \eqref{dH}, \eqref{de_H} and \eqref{de^H} that $\mathbf{d}^\mathcal{H}$ and $\mathbf{d}_\mathcal{H}$ are distances whenever $\mathbf{d}$ is a distance.

\begin{prp}\label{Hausfunc}
For any $\mathbf{d},\mathbf{e}\in[0,\infty]^{X\times X}$,
\begin{align}
\label{dH}\mathbf{d}_\mathcal{H}&\leq\mathbf{d}^\mathcal{H}.\\
\label{de_H}(\mathbf{d}\circ\mathbf{e})_\mathcal{H}&\leq\mathbf{d}_\mathcal{H}\circ\mathbf{e}_\mathcal{H}\precapprox(\mathbf{d}\circ\mathbf{e})_\mathcal{H}.\\
\label{de^H}(\mathbf{d}\circ\mathbf{e})^\mathcal{H}&\leq\mathbf{d}_\mathcal{H}\circ\mathbf{e}^\mathcal{H}\precapprox(\mathbf{d}\circ\mathbf{e})^\mathcal{H}.\\
\label{d^He^H}\mathbf{d}^\mathcal{H}\circ\mathbf{e}^\mathcal{H}&=\mathbf{d}^\mathcal{H}\circ\mathbf{e}_\mathcal{H}.
\end{align}
\end{prp}

\begin{proof}\
\begin{itemize}
\item[\eqref{dH}] $Y\mathbf{d}_\mathcal{H}Z=Y(\mathbf{d}Z)=\sup\limits_{y\in Y}\inf\limits_{z\in Z}y\mathbf{d}z\leq\sup\limits_{y\in Y}\inf\limits_{z\in Z}Y\mathbf{d}z=(Y\mathbf{d})Z=Y\mathbf{d}^\mathcal{H}Z$.

\item[\eqref{de_H}]  First note that, for any $W,Y,Z\subseteq X$,
\begin{align*}
Y(\mathbf{d}\circ\mathbf{e})_\mathcal{H}Z&=Y((\mathbf{d}\circ\mathbf{e})Z)\\
&=\sup_{y\in Y}\inf_{z\in Z}\inf_{x\in X}(y\mathbf{d}x+x\mathbf{e}z)\\
&=\sup_{y\in Y}\inf_{x\in X}(y\mathbf{d}x+x\mathbf{e}Z)\\
&\leq\sup_{y\in Y}\inf_{w\in W}(y\mathbf{d}w+w\mathbf{e}Z)\\
&\leq\sup_{y\in Y}(y\mathbf{d}W+W(\mathbf{e}Z))\\
&=Y(\mathbf{d}W)+W(\mathbf{e}Z)\\
&=Y\mathbf{d}_\mathcal{H}W+W\mathbf{e}_\mathcal{H}Z,
\end{align*}
i.e. $(\mathbf{d}\circ\mathbf{e})_\mathcal{H}\leq\mathbf{d}_\mathcal{H}\circ\mathbf{e}_\mathcal{H}$.  On the other hand, for any $r>Y((\mathbf{d}\circ\mathbf{e})Z)$ and $y\in Y$, we have $w_y\in X$ and $z\in Z$ with $y\mathbf{d}w_y+w_y\mathbf{e}z<r$.  For $W=\{w_y:y\in Y\}$ we then have $Y(\mathbf{d}W)+W(\mathbf{d}Z)\leq2r$ and hence
\[\mathbf{d}_\mathcal{H}\circ\mathbf{e}_\mathcal{H}\leq2(\mathbf{d}\circ\mathbf{e})_\mathcal{H}.\]

\item[\eqref{de^H}]  First note that, for any $W,Y,Z\subseteq X$,
\begin{align*}
&Y(\mathbf{d}\circ\mathbf{e})^\mathcal{H}Z\\
=\ &(Y(\mathbf{d}\circ\mathbf{e}))Z\\
\leq\ &\inf_{z\in Z}\sup_{y\in Y}\inf_{w\in W}(y\mathbf{d}w+w\mathbf{e}z)\\
\leq\ &\inf_{z\in Z}\sup_{y\in Y}(y\mathbf{d}W+W\mathbf{e}z)\\
=\ &Y(\mathbf{d}W)+(W\mathbf{e})Z\\
=\ &Y\mathbf{d}_\mathcal{H}W+W\mathbf{e}^\mathcal{H}Z,
\end{align*}
i.e. $(\mathbf{d}\circ\mathbf{e})^\mathcal{H}\leq\mathbf{d}_\mathcal{H}\circ\mathbf{e}^\mathcal{H}$.  On the other hand, for any $r>(Y(\mathbf{d}\circ\mathbf{e}))Z$, we have $z\in Z$ such that, for all $y\in Y$, there is some $w_y\in X$ with $y\mathbf{d}w_y+w_y\mathbf{e}z<r$.  For $W=\{w_y:y\in Y\}$ we then have $Y(\mathbf{d}W)+(W\mathbf{d})Z\leq2r$ and hence
\[\mathbf{d}_\mathcal{H}\circ\mathbf{e}^\mathcal{H}\leq2(\mathbf{d}\circ\mathbf{e})^\mathcal{H}.\]

\item[\eqref{d^He^H}]  By \eqref{dH}, we have $\mathbf{d}^\mathcal{H}\circ\mathbf{e}_\mathcal{H}\leq\mathbf{d}^\mathcal{H}\circ\mathbf{e}^\mathcal{H}$.  Conversely, for any $W,Y,Z\subseteq X$,
\begin{align*}
Y(\mathbf{d}^\mathcal{H}\circ\mathbf{e}^\mathcal{H})Z&\leq\inf_{w\in W}(Y\mathbf{d}^\mathcal{H}\{w\}+\{w\}\mathbf{e}^\mathcal{H}Z)\\
&=\inf_{w\in W}(Y\mathbf{d}w+w\mathbf{e}Z)\\
&\leq(Y\mathbf{d})W+W(\mathbf{e}Z)\\
&=Y\mathbf{d}^\mathcal{H}W+W\mathbf{e}_\mathcal{H}Z.\qedhere
\end{align*}
\end{itemize}
\end{proof}

\begin{prp}\label{HausdorffProp}
$\mathcal{P}(X)$ is $\mathbf{d}^\mathcal{H}$-$\max$-complete and $\mathbf{d}_\mathcal{H}$-$\sup$-complete.
\end{prp}

\begin{proof}
Note $\mathbf{d}^\mathcal{H}\mathcal{Y}=\mathbf{d}^\mathcal{H}(\bigcup\mathcal{Y})$, as
\[Z\mathbf{d}^\mathcal{H}\mathcal{Y}=\inf_{Y\in\mathcal{Y}}Z\mathbf{d}^\mathcal{H}Y=\inf_{Y\in\mathcal{Y}}(Z\mathbf{d})Y=(Z\mathbf{d})\bigcup\mathcal{Y}=Z\mathbf{d}^\mathcal{H}(\bigcup\mathcal{Y}).\]
So if $\mathcal{Y}\subseteq\mathcal{P}(X)$ is $\mathbf{d}^\mathcal{H}$-directed or just $\mathbf{d}^\mathcal{H}$-final then $\bigcup\mathcal{Y}=\mathbf{d}^\mathcal{H}$-$\max\mathcal{Y}$, by \cite[(10.5)]{Bice2019a}, i.e. $\mathcal{P}(X)$ is $\mathbf{d}^\mathcal{H}$-$\max$-complete.  Likewise $\mathcal{Y}\mathbf{d}_\mathcal{H}=(\bigcup\mathcal{Y})\mathbf{d}_\mathcal{H}$, as
\[\mathcal{Y}\mathbf{d}_\mathcal{H}Z=\sup_{Y\in\mathcal{Y}}Y\mathbf{d}_\mathcal{H}Z=\sup_{Y\in\mathcal{Y}}Y(\mathbf{d}Z)=\bigcup\mathcal{Y}(\mathbf{d}Z)=(\bigcup\mathcal{Y})\mathbf{d}_\mathcal{H}Z.\]
If $\mathcal{Y}\subseteq\mathcal{P}(X)$ is $\mathbf{d}_\mathcal{H}$-directed or just $\mathbf{d}_\mathcal{H}$-final then, for all $Z\in\mathcal{Y}$,
\[Z\mathbf{d}_\mathcal{H}\bigcup\mathcal{Y}=Z(\mathbf{d}\bigcup\mathcal{Y})\leq\inf_{Y\in\mathcal{Y}}Z(\mathbf{d}Y)=\inf_{Y\in\mathcal{Y}}Z\mathbf{d}_\mathcal{H}Y=Z\mathbf{d}_\mathcal{H}\mathcal{Y}=0,\]
i.e. $Z\leq^{\mathbf{d}_\mathcal{H}}\bigcup\mathcal{Y}$ and hence $\bigcup\mathcal{Y}=\mathbf{d}_\mathcal{H}$-$\sup\mathcal{Y}$, i.e. $\mathcal{P}(X)$ is $\mathbf{d}_\mathcal{H}$-$\sup$-complete.
\end{proof}

Note that $\leq^{\mathbf{d}_\mathcal{H}}$ is reflexive precisely on the $\mathbf{d}$-final subsets of $X$.  In particular, $\mathbf{d}_\mathcal{H}$ is a hemimetric when restricted to the $\mathbf{d}$-directed subsets, which we denote by
\[\mathcal{P}^\mathbf{d}(X)=\{Y\subseteq X:Y\text{ is $\mathbf{d}$-directed}\}.\]
In contrast, $\mathbf{d}^\mathcal{H}$ may not be a hemimetric on $\mathcal{P}^\mathbf{d}(X)$, even when $\mathbf{d}$ is a hemimetric.  But there is one special situation in which this occurs.

\begin{dfn}
We call $X$ \emph{$\mathbf{d}$-Noetherian} if every $\mathbf{d}$-Cauchy sequence is $\mathbf{d}^\mathrm{op}$-Cauchy.
\end{dfn}

Note that if $\leq$ is a partial order relation on $X$ (identified with its characteristic function) then $X$ is $\leq$-Noetherian iff every increasing sequence is eventually constant, i.e. iff $X$ is Noetherian (or `upwards well-ordered') in the usual sense.

\begin{prp}\label{d^Hhemi}
The following are equivalent.
\begin{enumerate}
\item\label{dN} $X$ is $\mathbf{d}$-Noetherian.
\item\label{dpCnet} Every $\mathbf{d}$-pre-Cauchy net in $X$ is $\mathbf{d}^\mathrm{op}$-Cauchy.
\item\label{dCsubseq} Every $\mathbf{d}$-Cauchy sequence in $X$ has a $\mathbf{d}^\mathrm{op}$-pre-Cauchy subnet.
\end{enumerate}
\end{prp}

\begin{proof}
We immediately see that \eqref{dpCnet} $\Rightarrow$ \eqref{dN} $\Rightarrow$ \eqref{dCsubseq}.  Conversely, say \eqref{dpCnet} fails, so we have $\mathbf{d}$-pre-Cauchy $(x_\lambda)\subseteq X$ that is not $\mathbf{d}^\mathrm{op}$-Cauchy.  Then $(x_\lambda)$ is not even $\mathbf{d}^\mathrm{op}$-pre-Cauchy, otherwise $(x_\lambda)$ would be $\mathbf{d}^\vee$-pre-Cauchy and hence $\mathbf{d}^\vee$-Cauchy, by \cite[Proposition 7.2]{Bice2019a}.  Thus we have 
\begin{align*}
\lim_\gamma x_\gamma\mathbf{d}(x_\lambda)&=0.\\
\epsilon=\limsup_\gamma\,(x_\lambda)\mathbf{d}x_\gamma&>0.
\end{align*}
Thus we can take $\lambda_1$ with
\begin{align*}
\lim_\gamma x_{\lambda_1}\mathbf{d}x_\gamma&<\epsilon/4.\\
\lim_\gamma\,x_\gamma\mathbf{d}x_{\lambda_1}&>\epsilon/2.
\end{align*}
Then we can take $\lambda_2$ with $x_{\lambda_1}\mathbf{d}x_{\lambda_2}<\epsilon/4$, $x_{\lambda_2}\mathbf{d}x_{\lambda_1}>\epsilon/2$ and
\begin{align*}
\lim_\gamma x_{\lambda_2}\mathbf{d}x_\gamma&<\epsilon/16.\\
\lim_\gamma\,x_\gamma\mathbf{d}x_{\lambda_2}&>\epsilon/2.
\end{align*}
Continuing in this way we obtain a sequence $x_n=x_{\lambda_n}$ such that
\begin{align*}
x_n\mathbf{d}x_{n+1}&<\epsilon/4^{n+1}\\
x_{n+1}\mathbf{d}x_n&>\epsilon/2.
\end{align*}
Thus $(x_n)$ is $\mathbf{d}$-Cauchy and, for any $m<n$, $x_m\mathbf{d}x_n<\epsilon/3$ and hence
\[x_{n+1}\mathbf{d}x_m\geq x_{n+1}\mathbf{d}x_n-x_m\mathbf{d}x_n>\epsilon/6,\]
so $(x_n)$ has no $\mathbf{d}^\mathrm{op}$-pre-Cauchy subnet, i.e. \eqref{dCsubseq} fails, completing the logical loop.
\end{proof}

\begin{prp}\label{dNothdH}
If $X$ is $\mathbf{d}$-Noetherian then $\mathbf{d}^\mathcal{H}$ is a hemimetric on $\mathcal{P}^\mathbf{d}(X)$.
\end{prp}

\begin{proof}
For any $Y\in\mathcal{P}^\mathbf{d}(X)$, we have $\mathbf{d}$-pre-Cauchy $(x_\lambda)\subseteq Y\leq^\mathbf{d}(x_\lambda)$, by \cite[(9.8)]{Bice2019a}.  By \eqref{dpCnet} above, $(x_\lambda)$ is $\mathbf{d}^\mathrm{op}$-(pre-)Cauchy so, for any $\epsilon>0$, we have $\gamma$ such that $(x_\lambda)\mathbf{d}x_\gamma<\epsilon$.  Thus, for all $y\in Y$, $y\mathbf{d}x_\gamma\leq y\mathbf{d}(x_\lambda)+(x_\lambda)\mathbf{d}x_\gamma<\epsilon$ and hence $Y\mathbf{d}^\mathcal{H}Y\leq Y\mathbf{d}x_\gamma<\epsilon$, i.e. $Y\mathbf{d}^\mathcal{H}Y=0$.
\end{proof}

Now we generalise the construction of a domain from an abstract basis.

\begin{thm}\label{predomaincompletion}
If $X$ is $\mathbf{d}$-$\max$-continuous then
\begin{equation}\label{dH=dH|}
\underline{\mathbf{d}^\mathcal{H}|_{\mathcal{P}^\mathbf{d}(X)}}\ =\ \mathbf{d}_\mathcal{H}|_{\mathcal{P}^\mathbf{d}(X)}.
\end{equation}
Moreover, $\mathcal{P}^\mathbf{d}(X)$ is a $\mathbf{d}^\mathcal{H}$-$\max$-domain with $\mathbf{d}^\mathcal{H}$-$\max$-basis $\{(\leq^\mathbf{d}x):x\in X\}$ and
\begin{equation}\label{dHxdy}
(\leq^\mathbf{d}x)\mathbf{d}^\mathcal{H}(\leq^\mathbf{d}y)\ \leq\ x\mathbf{d}y.
\end{equation}
\end{thm}

\begin{proof}\
\begin{itemize}
\item[\eqref{dH=dH|}]  As $\mathbf{d}\leq\mathbf{d}\circ\underline{\mathbf{d}}$, \eqref{de^H} and \eqref{d^He^H} yield $\mathbf{d}^\mathcal{H}\leq\mathbf{d}^\mathcal{H}\circ\underline{\mathbf{d}}_\mathcal{H}$ and hence $\underline{\mathbf{d}^\mathcal{H}}\leq\underline{\mathbf{d}}_\mathcal{H}\leq\mathbf{d}_\mathcal{H}$.  As $X$ is $\mathbf{d}$-$\max$-continuous, for all $x\in X$, $x=\mathbf{d}$-$\max(\leq^\mathbf{d}x)=\underline{\mathbf{d}}$-$\sup(\leq^\mathbf{d}x)$, by \cite[(9.2)]{Bice2019a}, so $x\underline{\mathbf{d}}=(\leq^\mathbf{d}x)\underline{\mathbf{d}}=(\leq^\mathbf{d}x)\mathbf{d}$, by \cite[(10.4)]{Bice2019a}.  Thus, for any $Y,Z\in\mathcal{P}^\mathbf{d}(X)$,
\begin{align*}
Y\mathbf{d}_\mathcal{H}Z=Y(\mathbf{d}Z)&=\sup_{y\in Y}\inf_{z\in Z}y\mathbf{d}z\\
&\leq\sup_{y\in Y}\inf_{x\in Y}\inf_{z\in Z}(y\mathbf{d}x+x\underline{\mathbf{d}}z)\\
&\leq\sup_{y\in Y}\inf_{x\in Y}y\mathbf{d}x+\sup_{x\in Y}\inf_{z\in Z}x\underline{\mathbf{d}}z\\
&=\sup_{x\in Y}\inf_{z\in Z}(\leq^\mathbf{d}x)\mathbf{d}z\\
&=\sup_{x\in Y}((\leq^\mathbf{d}x)\mathbf{d})Z\\
&=\sup_{x\in Y}((\leq^\mathbf{d}x)\mathbf{d}^\mathcal{H}Z-(\leq^\mathbf{d}x)\mathbf{d}^\mathcal{H}Y).
\end{align*}
As $(\leq^\mathbf{d}x)\in\mathcal{P}^\mathbf{d}(X)$, this shows that
\[\mathbf{d}_\mathcal{H}|_{\mathcal{P}^\mathbf{d}(X)}\leq\underline{\mathbf{d}^\mathcal{H}|_{\mathcal{P}^\mathbf{d}(X)}}\leq\underline{\mathbf{d}^\mathcal{H}}|_{\mathcal{P}^\mathbf{d}(X)}\leq\mathbf{d}_\mathcal{H}|_{\mathcal{P}^\mathbf{d}(X)}.\]
\end{itemize}

As $X$ is $\mathbf{d}$-$\max$-continuous, $(\leq^\mathbf{d}x)$ is $\mathbf{d}$-directed with $\mathbf{d}$-maximum $x$, for all $x\in X$.  Thus if $Y\in\mathcal{P}^\mathbf{d}(X)$ then $\mathcal{Y}=\{(\leq^\mathbf{d}y):y\in Y\}$ is $\mathbf{d}^\mathcal{H}$-directed.  Indeed if $G\in\mathcal{F}(Y)$ let $\mathcal{G}=\{(\leq^\mathbf{d}y):y\in G\}$ so
\begin{align*}
(\mathcal{G}\mathbf{d}^\mathcal{H})\mathcal{Y}&=\inf_{y\in Y}\sup_{z\in G}(\leq^\mathbf{d}z)\mathbf{d}^\mathcal{H}(\leq^\mathbf{d}y)=\inf_{y\in Y}\sup_{z\in G}((\leq^\mathbf{d}z)\mathbf{d})(\leq^\mathbf{d}y)\\
&\leq\inf_{y\in Y}\sup_{z\in G}z\mathbf{d}(\leq^\mathbf{d}y)=\inf_{y\in Y}\sup_{z\in G}z\mathbf{d}y=(G\mathbf{d})Y=0,
\end{align*}
as $Y$ is $\mathbf{d}$-directed.  For all $y\in Y$, $(\leq^\mathbf{d}y)\mathbf{d}^\mathcal{H}Y=((\leq^\mathbf{d}y)\mathbf{d})Y\leq(\leq^\mathbf{d}y)\mathbf{d}y=0$, i.e. $(\leq^\mathbf{d}y)\leq^{\mathbf{d}^\mathcal{H}}Y$.  Moreover, for all $Z\in\mathcal{P}^\mathbf{d}(X)$,
\begin{align*}
Z\mathbf{d}^\mathcal{H}\mathcal{Y}&=\inf_{y\in Y}Z\mathbf{d}^\mathcal{H}(\leq^\mathbf{d}y)=\inf_{y\in Y}(Z\mathbf{d})(\leq^\mathbf{d}y)=\inf_{x\leq^\mathbf{d}y\in Y}\sup_{z\in Z}z\mathbf{d}x\\
&\leq\inf_{w\in Y}\inf_{x\leq^\mathbf{d}y\in Y}\sup_{z\in Z}(z\mathbf{d}w+w\mathbf{d}x)=\inf_{w\in Y}\inf_{y\in Y}(Z\mathbf{d}w+w\mathbf{d}(\leq^\mathbf{d}y))\\
&=\inf_{w\in Y}\inf_{y\in Y}(Z\mathbf{d}w+w\mathbf{d}y)=\inf_{w\in Y}(Z\mathbf{d}w+w\mathbf{d}Y)\\
&\leq(Z\mathbf{d})Y+Y(\mathbf{d}Y)=(Z\mathbf{d})Y\\
&=Z\mathbf{d}^\mathcal{H}Y,
\end{align*}
i.e. $\mathbf{d}^\mathcal{H}\mathcal{Y}\leq\mathbf{d}^\mathcal{H}Y$ so $Y=\mathbf{d}^\mathcal{H}$-$\max\mathcal{Y}$.  Thus $\mathcal{P}^\mathbf{d}(X)$ is $\mathbf{d}^\mathcal{H}$-$\max$-continuous with $\mathbf{d}^\mathcal{H}$-$\max$-basis $\{(\leq^\mathbf{d}x):x\in X\}$.

If $\mathcal{Y}\subseteq\mathcal{P}^\mathbf{d}(X)$ is $\mathbf{d}^\mathcal{H}$-directed (or just $\mathbf{d}_\mathcal{H}$-directed) then $\bigcup\mathcal{Y}\in\mathcal{P}^\mathbf{d}(X)$, so $\mathcal{P}^\mathbf{d}(X)$ is $\mathbf{d}^\mathcal{H}$-$\max$-complete, as in \autoref{HausdorffProp}.  By \eqref{de^H}, $\mathbf{d}^\mathcal{H}\leq\mathbf{d}_\mathcal{H}\circ\mathbf{d}^\mathcal{H}$ so $\overline{\mathbf{d}^\mathcal{H}}\leq\mathbf{d}_\mathcal{H}$ and hence, by \eqref{dH=dH|},
\[\overline{\mathbf{d}^\mathcal{H}|_{\mathcal{P}^\mathbf{d}(X)}}\leq\overline{\mathbf{d}^\mathcal{H}}|_{\mathcal{P}^\mathbf{d}(X)}\leq\mathbf{d}_\mathcal{H}|_{\mathcal{P}^\mathbf{d}(X)}=\underline{\mathbf{d}^\mathcal{H}|_{\mathcal{P}^\mathbf{d}(X)}}.\]
Thus $\mathcal{P}^\mathbf{d}(X)$ is a $\mathbf{d}^\mathcal{H}$-$\max$-domain.

\begin{itemize}
\item[\eqref{dHxdy}]  As $X$ is $\mathbf{d}$-$\max$-continuous, for any $x,y\in X$, $y=\mathbf{d}$-$\max(\leq^\mathbf{d}y)$ so
\[(\leq^\mathbf{d}x)\mathbf{d}^\mathcal{H}(\leq^\mathbf{d}y)=((\leq^\mathbf{d}x)\mathbf{d})(\leq^\mathbf{d}y)\leq x\mathbf{d}(\leq^\mathbf{d}y)=x\mathbf{d}y.\qedhere\]
\end{itemize}
\end{proof}

\begin{cor}\label{pdcomp}
The following are equivalent.
\begin{enumerate}
\item\label{pdcompp} X is a $\mathbf{d}$-$\max$-predomain.
\item\label{pdcompb} X is a $\mathbf{d}'$-$\max$-basis of a $\mathbf{d}'$-$\max$-domain $X'\supseteq X$ with $\mathbf{d}'|_X=\mathbf{d}$.
\end{enumerate}
\end{cor}

\begin{proof}\
\begin{itemize}
\item[\eqref{pdcompp}$\Rightarrow$\eqref{pdcompb}]  Assume \eqref{pdcompp} and let $X'$ be the (disjoint) union of $X$ and $\mathcal{P}^\mathbf{d}(X)$.  Extend $\mathbf{d}^\mathcal{H}$ to $\mathbf{d}'$ on $X'$ by making each $x\in X$ $\mathbf{d}'$-equivalent to $(\leq^\mathbf{d}x)$.  By \autoref{predomaincompletion}, the only thing left to show is that the inequality in \eqref{dHxdy} is an equality.  For this note that, for any $x,y\in X$, $\overline{\mathbf{d}}\leq\underline{\mathbf{d}}$ implies
\begin{align}
\label{dHxdy=}(\leq^\mathbf{d}x)\mathbf{d}^\mathcal{H}(\leq^\mathbf{d}y)&\geq(\leq^\mathbf{d}x)\underline{\mathbf{d}}^\mathcal{H}(\leq^\mathbf{d}y)=((\leq^\mathbf{d}x)\underline{\mathbf{d}})(\leq^\mathbf{d}y)=x\underline{\mathbf{d}}(\leq^\mathbf{d}y)\\
\nonumber&\geq x\overline{\mathbf{d}}(\leq^\mathbf{d}y)=x\mathbf{d}(\leq^\mathbf{d}y)=x\mathbf{d}y.
\end{align}

\item[\eqref{pdcompb}$\Rightarrow$\eqref{pdcompp}]  If $X\subseteq X'$ is a $\mathbf{d}'$-$\max$-basis and $\mathbf{d}=\mathbf{d}'|_X$ then $X$ is certainly $\mathbf{d}$-$\max$-continuous.  If $X'$ is also a $\mathbf{d}'$-$\max$-(pre)domain then $\overline{\mathbf{d}}=\overline{\mathbf{d}'}|_X\leq\underline{\mathbf{d}'}|_X=\underline{\mathbf{d}}$, by \autoref{reflexbasis}, i.e. $X$ is a $\mathbf{d}$-$\max$-predomain.\qedhere
\end{itemize}
\end{proof}

In other words, \eqref{pdcompp}$\Rightarrow$\eqref{pdcompb} above says every $\mathbf{d}$-predomain $X$ has a completion $X'$.  If we want to identify $\mathbf{d}'$-equivalent points, we can restrict $\mathbf{d}^\mathcal{H}$ further to $\mathbf{d}$-ideals (i.e. $\overline{\mathbf{d}}^\bullet$-closed $\mathbf{d}$-directed subsets \textendash\, see \cite[Proposition 9.10]{Bice2019a}) denoted by
\[\mathcal{I}^\mathbf{d}(X)=\{I\subseteq X:I\text{ is a $\mathbf{d}$-ideal}\}.\]

\begin{thm}\label{predomainuniversality}
If $B$ is a $\mathbf{d}$-$\max$-basis of $\mathbf{d}$-$\max$-predomain $X$ then
\begin{equation}\label{isocomp}
x\mapsto(\leq^\mathbf{d}x)\cap B
\end{equation}
is an isometry (w.r.t. $\mathbf{d}$ on $X$ and $\mathbf{d}^\mathcal{H}$ on $\mathcal{I}^\mathbf{d}(X)$) to the $\mathbf{d}^\mathcal{H}$-$\max$-domain $\mathcal{I}^\mathbf{d}(B)$.  Moreover, this isometry is onto $\mathcal{I}^\mathbf{d}(B)$ iff $X$ is a $\mathbf{d}$-$\max$-domain.
\end{thm}

\begin{proof}
As $B$ is a $\mathbf{d}$-$\max$-basis, $(\leq^\mathbf{d}x)\cap B\in\mathcal{I}^\mathbf{d}(B)$, for all $x\in X$.  Every $Y\in\mathcal{P}^\mathbf{d}(B)$ is $\mathbf{d}^\mathcal{H}$-equivalent to $I_Y=\overline{\mathbf{d}}^\bullet\!\!$-$\mathrm{cl}(Y)\in\mathcal{I}^\mathbf{d}(B)$ so $\mathcal{I}^\mathbf{d}(B)$ is also a $\mathbf{d}^\mathcal{H}$-$\max$-domain and
\[((\leq^\mathbf{d}x)\cap B)\mathbf{d}^\mathcal{H}((\leq^\mathbf{d}y)\cap B)\ =\ x\mathbf{d}y,\]
i.e. \eqref{isocomp} is an isometry.  Also, as $B$ is a $\mathbf{d}$-$\max$-basis, for $Y\in\mathcal{P}^\mathbf{d}(X)$,
\[\text{$x=\mathbf{d}$-$\max Y$}\quad\Leftrightarrow\quad\mathbf{d}x=\mathbf{d}Y=\mathbf{d}I_Y\quad\Leftrightarrow\quad(\leq^\mathbf{d}x)\cap B=I_Y\cap B,\]
so \eqref{isocomp} is onto iff $X$ is $\mathbf{d}$-$\max$-complete and hence a $\mathbf{d}$-$\max$-domain.
\end{proof}

In other words $\mathcal{I}^\mathbf{d}(B)$ is universal among $\mathbf{d}$-$\max$-predomain extensions of $B$, and unique among $\mathbf{d}$-$\max$-domain extensions, up to isometry (and $\mathbf{d}$-equivalence).

At this point we could develop a parallel theory of Hausdorff distances on nets $\mathsf{N}(X)$ on $X$, specifically we could define
\begin{align*}
(y_\lambda)\mathbf{d}^\mathsf{H}(z_\gamma)&=((y_\lambda)\mathbf{d})(z_\gamma)=\liminf_\gamma\limsup_\lambda y_\lambda\mathbf{d}z_\gamma.\\
(y_\lambda)\mathbf{d}_\mathsf{H}(z_\gamma)&=(y_\lambda)(\mathbf{d}(z_\gamma))=\limsup_\lambda\liminf_\gamma y_\lambda \mathbf{d}z_\gamma.
\end{align*}
The analog of \autoref{Hausfunc} would be no problem, but completeness and continuity would involve nets of nets, which are technically challenging to work with.  Instead, to get topological analogs of the above results, we turn to formal balls.

\section{Formal Balls}\label{FormalBalls}\label{FB}

The following is based on \cite[Definition 7.3.1]{Goubault2013}, although the formal ball construction goes back to \cite{WeihrauchSchreiber1981}.

\begin{dfn}
Define $\mathbf{d}_+$ on $X_+=X\times[0,\infty)$ by
\[(x,r)\mathbf{d}_+(y,s)=(x\mathbf{d}y-r+s)_+.\]
\end{dfn}

This does not quite extend to a functor on $\mathbf{GRel}$, as $_+$ does not preserve identity morphisms.  Indeed, recall that we identify $=$ with its characteristic function, so
\[(x,r)\!=_+\!(y,s)\ =\ \begin{cases}(s-r)_+&\text{if }x=y\\ \infty &\text{if }x\neq y,\end{cases}\]
which is not (the characteristic function of) $=$ on $X_+$.  However, $_+$ does preserve composition.  In particular, this means $\mathbf{d}_+$ is a distance whenever $\mathbf{d}$ is.

\begin{prp}\label{Bfunc}
\[(\mathbf{d}\circ\mathbf{e})_+=\mathbf{d}_+\circ\mathbf{e}_+.\]
\end{prp}

\begin{proof}
For $(\mathbf{d}\circ\mathbf{e})_+=\mathbf{d}_+\circ\mathbf{e}_+$, note
\begin{align*}
(x,r)(\mathbf{d}_+\circ\mathbf{e}_+)(y,s)&=\inf_{z\in X,t\in\mathbb{R}_+}(x,r)\mathbf{d}_+(z,t)+(z,t)\mathbf{e}_+(y,s).\\
&=\inf_{z\in X,t\in\mathbb{R}_+}(x\mathbf{d}z-r+t)_++(z\mathbf{e}y-t+s)_+.\\
&=\inf_{z\in X,z\mathbf{e}y<\infty,t=z\mathbf{e}y+s}(x\mathbf{d}z-r+t)_++(z\mathbf{e}y-t+s)_+.\\
&=\inf_{z\in X}(x\mathbf{d}z+z\mathbf{e}y-r+s)_+.\\
&=(x(\mathbf{d}\circ\mathbf{e})y-r+s)_+\\
&=(x,r)(\mathbf{d}\circ\mathbf{e})_+(y,s).\qedhere
\end{align*}
\end{proof}

As $\mathbf{d}\leq\mathbf{d}\circ\underline{\mathbf{d}}$ and $\mathbf{d}\leq\overline{\mathbf{d}}\circ\mathbf{d}$, it follows that $\mathbf{d}_+\leq\mathbf{d}_+\circ\underline{\mathbf{d}}_+$ and $\mathbf{d}_+\leq\overline{\mathbf{d}}_+\circ\mathbf{d}_+$ so
\[\underline{\mathbf{d}_+}\leq\underline{\mathbf{d}}_+\qquad\text{and}\qquad\overline{\mathbf{d}_+}\leq\overline{\mathbf{d}}_+.\]

However, the reverse inequality can fail, e.g. for the right projection distance $\mathbf{d}$ given at the end of \cite[\S8]{Bice2019a}.  Specifically, define $\mathbf{d}$ on $X=[0,\infty)$ by $y\mathbf{d}z=z$ so $X_+=[0,\infty)\times[0,\infty)$ and, for all $x,y,r,s,t\in[0,\infty)$ with $t\leq s$,
\[(x,r)\mathbf{d}_+(y,s)=(y-r+s)_+=(x,r)\mathbf{d}_+(y+t,s-t).\]
This means $\underline{\mathbf{d}_+}$ is not a quasimetric, as it identifies all pairs of the form $(y,s)$ and $(y+t,s-t)$.  However, $y\underline{\mathbf{d}}z=(z-y)_+$, which is just the opposite of the usual quasimetric on $[0,\infty)$ and hence $\underline{\mathbf{d}}_+$ is also a quasimetric \textendash\, see \cite[Exercise 7.3.7]{Goubault2013} \textendash\, so, in particular, $\underline{\mathbf{d}_+}\neq\underline{\mathbf{d}}_+$.

However, this example is very far from being $\mathbf{d}^\bullet_\circ$-continuous.  In fact, this anomaly disappears if $X$ is merely $\mathbf{d}$-initial, i.e. $\mathbf{0}\circ\mathbf{d}=\mathbf{0}$.

\begin{prp}
\[\mathbf{0}\circ\mathbf{d}=\mathbf{0}\qquad\Rightarrow\qquad\underline{\mathbf{d}_+}=\underline{\mathbf{d}}_+\qquad\Leftrightarrow\qquad\forall y,z\in X\ (\sup_{x\in X}(x\mathbf{d}y-x\mathbf{d}z)\geq0).\]
\end{prp}

\begin{proof}
For any $y,z\in X$, if $\mathbf{0}\circ\mathbf{d}=\mathbf{0}$ then $\inf_{x\in X}x\mathbf{d}y$ and hence
\[\sup_{x\in X}(x\mathbf{d}y-x\mathbf{d}z)\geq\sup_{x\in X}(-x\mathbf{d}y)=-\inf_{x\in X}x\mathbf{d}y=0.\]
Thus it suffices to prove the last $\Leftrightarrow$.  For any $y,z\in X$ and $s,t\in[0,\infty)$,
\begin{align*}
(z,t)\underline{\mathbf{d}_+}(y,s)&=\sup_{x\in X,r\geq0}((x,r)\mathbf{d}_+(y,s)-(x,r)\mathbf{d}_+(z,t))_+\\
&=\sup_{x\in X,r\geq0}((x\mathbf{d}y-r+s)_+-(x\mathbf{d}z-r+t)_+)_+\\
&=\sup_{x\in X,r=x\mathbf{d}z+t}((x\mathbf{d}y-r+s)_+-(x\mathbf{d}z-r+t)_+)_+\\
&=(\sup_{x\in X}(x\mathbf{d}y-x\mathbf{d}z)-t+s)_+.
\end{align*}
On the other hand,
\[(z,t)\underline{\mathbf{d}}_+(y,s)=(z\underline{\mathbf{d}}y-t+s)_+=(\sup_{x\in X}(x\mathbf{d}y-x\mathbf{d}z)_+-t+s)_+.\]
So if $\sup_{x\in X}(x\mathbf{d}y-x\mathbf{d}z)\geq0$ then these two expressions coincide, otherwise taking $s=-\sup_{x\in X}(x\mathbf{d}y-x\mathbf{d}z)>0$ yields
\[(z,0)\underline{\mathbf{d}_+}(y,s)=0<s=(z,0)\underline{\mathbf{d}}_+(y,s).\qedhere\]
\end{proof}

Formal balls were originally introduced just as order structures $(X_+,\leq^{\mathbf{d}_+})$ with the primary purpose of reducing metric theory to order theory.  Indeed, we can always recover $\mathbf{d}$ from the preorder $\leq^{\mathbf{d}_+}$ or even the strict order $<^{\mathbf{d}_+}$ (see \eqref{StrictOrder} above) so this reduction is always possible, at least in principle.

\begin{prp}\label{xdy}
For any $x,y\in X$,
\begin{align*}
x\mathbf{d}y&=\min\{r\in\mathbb{R}_+:(x,r)\leq^{\mathbf{d}_+}(y,0)\}\\
&=\inf\{r\in\mathbb{R}_+:(x,r)<^{\mathbf{d}_+}(y,0)\}\qquad\text{if }\underline{\mathbf{d}_+}=\underline{\mathbf{d}}_+.
\end{align*}
\end{prp}

\begin{proof}
This follows directly from
\begin{align}
\label{<=Bd}(x,r)\leq^{\mathbf{d}_+}(y,s)\quad&\Leftrightarrow\quad x\mathbf{d}y\leq r-s.\\
\label{<Bd}(x,r)<^{\mathbf{d}_+}(y,s)\quad&\Leftrightarrow\quad x\mathbf{d}y<r-s\qquad\text{if }\underline{\mathbf{d}_+}=\underline{\mathbf{d}}_+.
\end{align}
Indeed, \eqref{<=Bd} is immediate from the definitions.  For \eqref{<Bd}, say $\epsilon=r-s-x\mathbf{d}y>0$ and $(y,s)<^{\underline{\mathbf{d}_+}}_\epsilon(z,t)$, so $(y,s)\underline{\mathbf{d}}_+(z,t)<\epsilon$, as $\underline{\mathbf{d}_+}=\underline{\mathbf{d}}_+$.  Then $y\underline{\mathbf{d}}z-s+t<\epsilon=r-s-x\mathbf{d}y$ so $x\mathbf{d}z\leq x\mathbf{d}y+y\underline{\mathbf{d}}z<r-t$ and hence $(x,r)\leq^{\mathbf{d}_+}(z,t)$.  Thus $(x,r)<^{\mathbf{d}_+}(y,s)$.  Conversely, if $\epsilon>0$ and $(x,r)\leq^{\mathbf{d}_+}(z,t)$, for all $(z,t)$ with $(y,s)\underline{\mathbf{d}}_+(z,t)<\epsilon$ then, in particular, $(x,r)\leq^{\mathbf{d}_+}(y,s+\frac{1}{2}\epsilon)$ so $x\mathbf{d}y\leq r-s-\frac{1}{2}\epsilon<r-s$.
\end{proof}

What sets $\mathbf{d}_+$ apart from other distances is interpolation.

\begin{prp}
If $\underline{\mathbf{d}_+}=\underline{\mathbf{d}}_+$ then
\begin{align}
\label{+PInterpolation}=_+\!\circ{}<^{\mathbf{d}_+}\!\!\mathcal{P}\ &=\ \mathbf{d}_+\mathcal{P}.\\
\label{+Interpolation}<^{=_+}\circ<^{\mathbf{d}_+}\ &=\ \ <^{\mathbf{d}_+}\!\!.
\end{align}
\end{prp}

\begin{proof}\
\begin{itemize}
\item[\eqref{+PInterpolation}] For any $Y\subseteq X_+$, \eqref{<Bd} yields
\begin{align*}
(x,r)(=_+\!\circ{}<^{\mathbf{d}_+}\!\!\mathcal{P})Y&=\inf\{(x,r)\!=_+\!(x,t):\forall(y,s)\in Y\ (x,t)<^{\mathbf{d}_+}(y,s)\}\\
&=\inf\{(t-r)_+:\sup_{(y,s)\in Y}(x\mathbf{d}y+s)<t\}\\
&=\sup_{(y,s)\in Y}(x\mathbf{d}y+s-r)_+\\
&=\sup_{(y,s)\in Y}(x,r)\mathbf{d}_+(y,s)\\
&=(x,r)(\mathbf{d}_+\mathcal{P})Y.
\end{align*}

\item[\eqref{+Interpolation}] If $(x,r)<^{\mathbf{d}_+}(y,s)$ then $x\mathbf{d}y<r-s$ so taking $t\in(x\mathbf{d}y+s,r)$ yields
\[(x,r)<^{=_+}(x,t)<^{\mathbf{d}_+}(y,s),\]
while if $(x,r)<^{=_+}(x,t)<^{\mathbf{d}_+}(y,s)$ then $t<r$ so $(x,r)<^{\mathbf{d}_+}(y,s)$. \qedhere
\end{itemize}
\end{proof}

These strong interpolation conditions are really what makes the formal ball construction so useful.  For example, as noted after \autoref{Bfunc}, $\overline{\mathbf{d}_+}\leq\overline{\mathbf{d}}_+$ so \eqref{+PInterpolation} (restricted to singletons on the right hand side) yields
\[(\overline{\mathbf{d}_+}\,\circ<^{\mathbf{d}_+})\leq(\overline{\mathbf{d}}_+\,\circ<^{\mathbf{d}_+})\leq(=_+\circ<^{\mathbf{d}_+})\leq\mathbf{d}_+.\]
This is precisely the condition required for \cite[Proposition 5.4]{Bice2019a}, which yields
\[\underline{<^{\mathbf{d}_+}}\ \ =\ \ \leq^{\underline{\mathbf{d}_+}}\ \ =\ \ \leq^{\underline{\mathbf{d}}_+}.\]
It is also the condition required for \cite[(10.9)]{Bice2019a} so, for all $x\in X$ and $Y\subseteq X$,
\begin{equation}\label{<max=>dmax}
x=\text{$<^{\mathbf{d}_+}$-$\max Y$}\qquad\Rightarrow\qquad x=\text{$\mathbf{d}_+$-$\max Y$}.
\end{equation}
On the other hand, \eqref{+Interpolation} yields $<^{\mathbf{d}_+}\ =\ <^{=_+}\circ<^{\mathbf{d}_+}\ \subseteq\ <^{\overline{\mathbf{d}_+}}\circ\leq^{\mathbf{d}_+}$.  This is precisely the condition required for \cite[(10.10)]{Bice2019a}, which yields the converse
\begin{equation}\label{<max<=dmax}
x=\text{$<^{\mathbf{d}_+}$-$\max Y$}\qquad\Leftarrow\qquad x=\text{$\mathbf{d}_+$-$\max Y$}.
\end{equation}
Indeed, with these interpolation conditions at our disposal, we can reduce Smyth completeness and continuity to their order theoretic counterparts in $X_+$.

\begin{thm}\label{contdomballs}
\begin{align}
\label{Bcomp}X\text{ is $\mathbf{d}^\bullet_\circ$-complete}\quad&\Leftrightarrow\quad\text{$X_+$ is $<^{\mathbf{d}_+}$-$\max$-complete,}\quad\text{if }\underline{\mathbf{d}_+}=\underline{\mathbf{d}}_+.\\
\label{Bcont}X\text{ is $\mathbf{d}^\bullet_\circ$-continuous}\quad&\Leftrightarrow\quad\text{$X_+$ is $<^{\mathbf{d}_+}$-$\max$-continuous and }\mathbf{0}\circ\mathbf{d}=\mathbf{0}.
\end{align}
\end{thm}

\begin{proof}\
\begin{itemize}
\item[\eqref{Bcomp}]  Assume $X_+$ is $<^{\mathbf{d}_+}$-$\max$-complete.  For any $\mathbf{d}$-Cauchy $(x_\lambda)\subseteq X$, define
\begin{equation}\label{Cideal}
I=\{(y,r):y\mathbf{d}(x_\lambda)<r\}.
\end{equation}
If $(y,r),(z,s)\in I$ then we can take positive $t<(r-y\mathbf{d}(x_\lambda)),(s-z\mathbf{d}(x_\lambda))$.  Then $y\mathbf{d}(x_\lambda)<r-t$ and $z\mathbf{d}(x_\lambda)<s-t$ so, for sufficiently large $\lambda$, $(y,r),(z,s)<^{\mathbf{d}_+}(x_\lambda,t)\in I$, as $(x_\lambda)$ is $\mathbf{d}$-Cauchy, i.e. $I$ is a $<^{\mathbf{d}_+}$-ideal with $\inf_{(y,r)\in I}r=0$.  As $X_+$ is $<^{\mathbf{d}_+}$-$\max$-complete, $I$ has a $<^{\mathbf{d}_+}$-maximum $(x,0)$, which is also a $\mathbf{d}_+$-maximum by \eqref{<max=>dmax}.  If $z\in X$, \cite[(9.2)]{Bice2019a} yields
\begin{align*}
z\mathbf{d}x&=(z,0)\mathbf{d}_+(x,0)=(z,0)\mathbf{d}_+I=(z,0)\overline{\mathbf{d}_+}I\leq(z,0)\overline{\mathbf{d}}_+I\\
&\leq\inf_{(z,r)\in I}(z,0)\overline{\mathbf{d}}_+(z,r)=\inf_{(z,r)\in I}r=z\mathbf{d}(x_\lambda)\\
&\leq\inf_{y\in X}(z\mathbf{d}y+y\mathbf{d}(x_\lambda))=\inf_{y\mathbf{d}(x_\lambda)<r}(z\mathbf{d}y+r)\\
&=(z,0)\mathbf{d}_+I=(z,0)\mathbf{d}_+(x,0)=z\mathbf{d}x,
\end{align*}
i.e. $z\mathbf{d}x=z\mathbf{d}(x_\lambda)$ so $x_\lambda\barrowc x$ and hence $X$ is $\mathbf{d}^\bullet_\circ$-complete.

Now assume $X$ is $\mathbf{d}^\bullet_\circ$-complete.  Any $<^{\mathbf{d}_+}$-directed $I\subseteq X_+$ yields a net
\[(x_{(x,r)})_{(x,r)\in I}.\]
By replacing each $(y,s)\in I$ with $(y,s-\inf_{(x,r)\in I}r)$ if necessary, we may assume $\inf_{(x,r)\in I}r=0$.  If $(y,s)<^{\mathbf{d}_+}(x,r)$ then $y\mathbf{d}x<s-r\leq s$ so $(x_{(x,r)})_{(x,r)\in I}$ is $\mathbf{d}$-Cauchy and $y\mathbf{d}(x_{(x,r)})\leq s$, for any $(y,s)\in I$.  Thus we have $z\in X$ with $x_{(x,r)\in I}\barrowc z$ and hence, for any $(y,s)\in I$, $y\mathbf{d}z\leq s$ so $(y,s)\leq^{\mathbf{d}_+}(z,0)$.  But for every $(y,s)\in I$, we have $(x,r)\in I$ with $(y,s)<^{\mathbf{d}_+}(x,r)\leq^{\mathbf{d}_+}(z,0)$ so $(y,s)<^{\mathbf{d}_+}(z,0)$, by \cite[(5.3)]{Bice2019a}, i.e. $I<^{\mathbf{d}_+}(z,0)$.  On the other hand, if $(y,s)<^{\mathbf{d}_+}(z,0)$ then $y\mathbf{d}(x_{(x,r)})=y\mathbf{d}z<s$, so we have $(x,r)\in I$ with $r<\frac{1}{2}(s-y\mathbf{d}z)$ and $y\mathbf{d}x<\frac{1}{2}(s+y\mathbf{d}z)<s-r$, i.e. $(y,s)<^{\mathbf{d}_+}(x,r)$.  Thus $(z,0)=\ <^{\mathbf{d}_+}$-$\max I$ so $X_+$ is $<^{\mathbf{d}_+}$-$\max$-complete.

\item[\eqref{Bcomp}]  Alternative proof: First we claim that
\[X\text{ is $\mathbf{d}^\bullet_\circ$-complete}\quad\Leftrightarrow\quad\text{$X_+$ is $\mathbf{d}_+{}^\bullet_\circ$-complete}.\]
For assume that $X$ is $\mathbf{d}^\bullet_\circ$-complete and take $\mathbf{d}_+$-Cauchy $(x_\lambda,r_\lambda)$.  In particular $(r_\lambda)$ is $\mathbf{q}^\mathrm{op}$-Cauchy (where $r\mathbf{q}s=(r-s)_+$) and bounded below by $0$, and hence $r_\lambda\rightarrow r$, for some $r\in[0,\infty)$.  This implies that $(x_\lambda)$ is also $\mathbf{d}$-Cauchy and hence $x_\lambda\barrowc x$, for some $x\in X$.  Thus $(x_\lambda,r_\lambda)\barrowc(x,r)$ so $X_+$ is $\mathbf{d}_+{}^\bullet_\circ$-complete.  Conversely, if $X_+$ is $\mathbf{d}_+{}^\bullet_\circ$-complete then any $\mathbf{d}$-Cauchy $(x_\lambda)$ yields $\mathbf{d}_+$-Cauchy $(x_\lambda,0)\barrowc(x,0)$ and hence $x_\lambda\barrowc x$, i.e. $X$ is $\mathbf{d}^\bullet_\circ$-complete.

Thus the claim is proved, and we next claim that
\[\qquad X_+\text{ is $\mathbf{d}_+{}^\bullet_\circ$-complete}\quad\Leftrightarrow\quad\text{$X_+$ is $<^{\mathbf{d}_+}$-$\max$-complete,}\quad\text{if }\underline{\mathbf{d}_+}=\underline{\mathbf{d}}_+.\]
Indeed, if $X_+$ is $\mathbf{d}_+{}^\bullet_\circ$-complete then $X_+$ is $\mathbf{d}_+$-$\max$-complete, by \cite[(11.2)]{Bice2019a}.  In particular, any $<^{\mathbf{d}_+}$-directed $Y\subseteq X_+$ has a $\mathbf{d}_+$-maximum, which is also a $<^{\mathbf{d}_+}$-maximum, by \eqref{<max=>dmax}, i.e. $X_+$ is $<^{\mathbf{d}_+}$-$\max$-complete.  Conversely, if $X_+$ is $<^{\mathbf{d}_+}$-$\max$-complete then any $<^{\mathbf{d}_+}$-directed $Y\subseteq X_+$ has a $<^{\mathbf{d}_+}$-maximum, which is also a $\mathbf{d}_+$-maximum, by \eqref{<max<=dmax}.  Thus $X_+$ is $<^{\mathbf{d}_+}$-$(\mathbf{d}_+$-$\max)$-complete and hence $\mathbf{d}_+{}^\bullet_\circ$-complete, by \cite[(11.5)]{Bice2019a}, which can be applied because \eqref{+PInterpolation} yields
\begin{equation}\label{d+PInterpolation}
(\underline{\mathbf{d}_+}\,\circ\leq^{\mathbf{d}_+\mathcal{P}})=(\underline{\mathbf{d}}_+\,\circ\leq^{\mathbf{d}_+}\!\mathcal{P})\leq(=_+\circ<^{\mathbf{d}_+}\!\!\mathcal{P})\leq\mathbf{d}_+\mathcal{P}.
\end{equation}

\item[\eqref{Bcont}] Assume $X$ is $\mathbf{d}^\bullet_\circ$-continuous.  So for each $x\in X$, we have $\mathbf{d}$-Cauchy $(x_\lambda)$ with $x_\lambda\carrowb x$ and hence $x_\lambda\mathbf{d}x\rightarrow0$, i.e. $\mathbf{0}\circ\mathbf{d}=\mathbf{0}$.  Now take $F\in\mathcal{F}(X_+)$ and $(y,s)\in X_+$ with $(x,r)<^{\mathbf{d}_+}(y,s)$, for all $(x,r)\in F$.  Thus we have $\epsilon>0$ with $x\mathbf{d}y<r-s-\epsilon$, for all $(x,r)\in F$.  \autoref{dbhcont} then yields $z\in X$ with $z\mathbf{d}y<\frac{1}{2}\epsilon$ and, for all $(x,r)\in F$, $x\mathbf{d}z<x\mathbf{d}y+\frac{1}{2}\epsilon<r-s-\frac{1}{2}\epsilon$ and hence $(x,r)<^{\mathbf{d}_+}(z,s+\frac{1}{2}\epsilon)<^{\mathbf{d}_+}(y,s)$, i.e. $X_+$ is $<^{\mathbf{d}_+}$-continuous.

Now assume $\mathbf{0}\circ\mathbf{d}=\mathbf{0}$ and $X_+$ is $<^{\mathbf{d}_+}$-$\max$-continuous.  Take $F\in\mathcal{F}(X)$, $y\in X$ and $\epsilon>0$.  As $\mathbf{0}\circ\mathbf{d}=\mathbf{0}$, we may enlarge $F$ if necessary and assume $w\mathbf{d}y<\epsilon$, for some $w\in F$.
 For all $x\in F$, $(x,x\mathbf{d}y+\epsilon)<^{\mathbf{d}_+}(y,0)$ so $<^{\mathbf{d}_+}$-$\max$-continuity yields $(z,r)\in X_+$ such that, for all $x\in F$, $(x,x\mathbf{d}y+\epsilon)<^{\mathbf{d}_+}(z,r)<^{\mathbf{d}_+}(y,0)$, i.e. $x\mathbf{d}z<x\mathbf{d}y+\epsilon-r$ and $z\mathbf{d}y<r$.  In particular, $0\leq w\mathbf{d}z<w\mathbf{d}y+\epsilon-r\leq2\epsilon-r$ so $z\mathbf{d}y<r<2\epsilon$ and $\max\limits_{x\in F}x\mathbf{d}z<\max\limits_{x\in F}x\mathbf{d}y+\epsilon$, i.e. $\mathcal{F}\mathbf{d}\circ\Phi^\mathbf{d}\leq\mathcal{F}\mathbf{d}$ so $X$ is $\mathbf{d}^\bullet_\circ$-continuous.

\item[\eqref{Bcont}]  Alternative proof: First we claim that
\begin{equation}\label{X+cts}
X\text{ is $\mathbf{d}^\bullet_\circ$-continuous}\quad\Leftrightarrow\quad\text{$X_+$ is $\mathbf{d}_+{}^\bullet_\circ$-continuous and }\mathbf{0}\circ\mathbf{d}=\mathbf{0}.
\end{equation}
For assume $X$ is $\mathbf{d}^\bullet_\circ$-continuous so, in particular, $\mathbf{0}\circ\mathbf{d}=\mathbf{0}$.  Also, for any $(x,r)\in X_+$, we have $\mathbf{d}$-Cauchy $(x_\lambda)\subseteq X$ with $x_\lambda\barrowc x$, which yields $\mathbf{d}_+$-Cauchy $(x_\lambda,r)\barrowc(x,r)$, i.e. $X_+$ is $\mathbf{d}_+{}^\bullet_\circ$-continuous.  Conversely, assume $X_+$ is $\mathbf{d}_+{}^\bullet_\circ$-continuous and $\mathbf{0}\circ\mathbf{d}=\mathbf{0}$.  Thus, for any $x\in X$, we have $\mathbf{d}_+$-Cauchy $(x_\lambda,r_\lambda)\barrowc(x,0)$ and, for any $\epsilon>0$, we have $y\in X$ with $y\mathbf{d}x<\epsilon$ and hence
\[\lim_\lambda r_\lambda\leq\lim_\lambda(y,0)\mathbf{d}_+(x_\lambda,r_\lambda)=(y,0)\mathbf{d}_+(x,0)=y\mathbf{d}x<\epsilon.\]
Thus $r_\lambda\rightarrow0$ and hence $x_\lambda\barrowc x$, i.e. $X$ is $\mathbf{d}^\bullet_\circ$-continuous.

Thus the claim is proved, and we next claim that
\[X_+\text{ is $\mathbf{d}_+{}^\bullet_\circ$-continuous}\quad\Leftrightarrow\quad\text{$X_+$ is $<^{\mathbf{d}_+}$-$\max$-continuous}.\]
Indeed, if $X_+$ is $\mathbf{d}_+{}^\bullet_\circ$-continuous then $X_+$ is $<^{\mathbf{d}_+}$-$(\mathbf{d}_+$-$\max)$-continuous, by \autoref{ctscor} \eqref{ctscor1} and \eqref{d+PInterpolation}, and hence $<^{\mathbf{d}_+}$-$\max$-continuous, by \eqref{<max<=dmax}.  Conversely, if $X_+$ is $<^{\mathbf{d}_+}$-$\max$-continuous then $X_+$ is $\mathbf{d}_+$-$\max$-continuous, by \eqref{<max=>dmax}, and hence $\mathbf{d}_+{}^\bullet_\circ$-continuous, by \autoref{maxctsequiv}.\qedhere
\end{itemize}
\end{proof}

Combining these yields an analogous result for domains.

\begin{thm}\label{KW}
\[X\text{ is a $\mathbf{d}^\bullet_\circ$-domain with }\mathbf{e}=\underline{\mathbf{d}}\quad\Leftrightarrow\quad X_+\text{ is a $<^{\mathbf{d}_+}$-$\max$-domain with }\leq^{\mathbf{e}_+}=\underline{<^{\mathbf{d}_+}}.\]
\end{thm}

\begin{proof}
Assume $X$ is a $\mathbf{d}^\bullet_\circ$-domain.  In particular, $\overline{\mathbf{d}}\leq\underline{\mathbf{d}}$ so $\overline{\mathbf{d}_+}\leq\overline{\mathbf{d}}_+\leq\underline{\mathbf{d}}_+$ and then \cite[Proposition 5.2]{Bice2019a} yields $<^{\mathbf{d}_+}\ =\ \ \leq^{\overline{\mathbf{d}_+}}\circ<^{\mathbf{d}_+}\ \supseteq\ \ \leq^{\underline{\mathbf{d}}_+}\circ<^{\mathbf{d}_+}$.  As $X$ is $\mathbf{d}^\bullet_\circ$-continuous and hence $\mathbf{d}\leq\mathbf{0}\circ\mathbf{d}$, this yields
\[\underline{<^{\mathbf{d}_+}}\ =\ \ \leq^{\underline{\mathbf{d}}_+}\ \ \subseteq\ \ \overline{<^{\mathbf{d}_+}}.\]
Thus $X_+$ is a $<^{\mathbf{d}_+}$-$\max$-domain, by \autoref{contdomballs}.

Conversely, say $X_+$ is a $<^{\mathbf{d}_+}$-$\max$-domain and $\leq^{\mathbf{e}_+}=\underline{<^{\mathbf{d}_+}}$.  We claim this implies $\mathbf{0}\circ\mathbf{d}=\mathbf{0}$.  To see this note that, as $X$ is $<^{\mathbf{d}_+}$-continuous, for any $x\in X$, we have $<^{\mathbf{d}_+}$-directed $Y\subseteq X_+$ such that $(x,0)=\ <^{\mathbf{d}_+}$-$\max Y$.  As $\leq^{\mathbf{e}_+}=\underline{<^{\mathbf{d}_+}}$, it follows that $Y$ is $\leq^{\mathbf{e}_+}$-directed and $(x,0)=\ \leq^{\mathbf{e}_+}$-$\sup Y$, by \cite[(10.4)]{Bice2019a}.  Let
\[\epsilon=\inf\{r\in[0,\infty):(y,r)\in Y\}.\]
We claim that $\epsilon=0$.  If not, $Z=\{(y,r-\epsilon):(y,r)\in Y\}$ is also $\leq^{\mathbf{e}_+}$-directed and hence $\leq^{\mathbf{e}_+}$-$\sup$-completeness (see \autoref{Rdomaineqs}) yields $(z,s)=\ \leq^{\mathbf{e}_+}$-$\sup Z$.  In particular, for all $(y,r)\in Y$, $(y,r-\epsilon)\leq^{\mathbf{e}_+}(z,s)$ and hence $(y,r)\leq^{\mathbf{e}_+}(z,s+\epsilon)$.  Thus $(x,0)\leq^{\mathbf{e}_+}(z,s+\epsilon)$, i.e. $0\leq x\mathbf{d}z\leq-s-\epsilon<0$, a contradiction.  This proves $\epsilon=0$ so we have $(y,r)\in Y$ with arbitrarily small $r$.  But $(y,r)<^{\mathbf{d}_+}(x,0)$ and hence $(y,r)\leq^{\mathbf{d}_+}(x,0)$, i.e. $y\mathbf{d}x\leq r$.  Thus $\mathbf{0}\circ\mathbf{d}=\mathbf{0}$, as claimed.

It follows that $\leq^{\mathbf{e}_+}\ =\ \underline{<^{\mathbf{d}_+}}=\ \leq^{\underline{\mathbf{d}}_+}$ and hence $\mathbf{e}=\underline{\mathbf{d}}$, by \autoref{xdy}.  As $X$ is a $<^{\mathbf{d}_+}$-$\max$-domain, we also have $\overline{<^{\mathbf{d}_+}}\ \supseteq\ \underline{<^{\mathbf{d}_+}}\ =\ \ \leq^{\underline{\mathbf{d}}_+}$ and hence
\[<^{\mathbf{d}_+}\ \supseteq\ \ \leq^{\underline{\mathbf{d}}_+}\circ<^{\mathbf{d}_+}\ \supseteq\ \ <^{(\underline{\mathbf{d}}_+\!\circ\mathbf{d}_+)}\ =\ \ <^{(\underline{\mathbf{d}}\circ\mathbf{d})_+}\]
For the last inclusion, note that if $(x,r)(\underline{\mathbf{d}}_+\!\circ\mathbf{d}_+)(y,s)$ then $x(\underline{\mathbf{d}}_+\!\circ\mathbf{d}_+)y<r-s$, so we have $\epsilon>0$ and $z\in X$ with $x\underline{\mathbf{d}}_+z+z\mathbf{d}_+y<r-s-\epsilon$ and hence $(x,r)\leq^{\underline{\mathbf{d}}_+}(z,s+z\mathbf{d}_+y+\epsilon)<^{\mathbf{d}_+}(y,s)$.  Thus \autoref{xdy} again yields $\mathbf{d}\leq\underline{\mathbf{d}}\circ\mathbf{d}$ and hence $\overline{\mathbf{d}}\leq\underline{\mathbf{d}}$.  Thus $X$ is a $\mathbf{d}^\bullet_\circ$-domain, again by \autoref{contdomballs}.
\end{proof}
 
\autoref{KW} can be considered as both a dual version of \cite[Theorem 9.1]{KostanekWaszkiewicz2011} and an extension of the Romaguera-Valero theorem characterising Smyth completeness for hemimetrics \textendash\, see \cite[Theorem 3.2]{RomagueraValero2010} or \cite[Theorem 7.3.11]{Goubault2013}.  Indeed, when $\mathbf{d}$ is a hemimetric, $X$ is trivially a $\mathbf{d}^\bullet_\circ$-predomain.  In particular, $\mathbf{0}\circ\mathbf{d}=\mathbf{0}$ so $\underline{<^{\mathbf{d}_+}}=\ \leq^{\mathbf{d}_+}$ and hence \autoref{KW} reduces to
\[X\text{ is $\mathbf{d}^\bullet_\circ$-complete}\qquad\Leftrightarrow\qquad X_+\text{ is a $<^{\mathbf{d}_+}$-$\max$-domain}.\]

\section{Smyth Completions}

As in \cite[Definition 7.5.2]{Goubault2013}, define the \emph{aperture} of $Y\subseteq X_+$ by
\[\alpha(Y)=\inf_{(x,r)\in Y}r.\]
Also denote the (directed/ideal) subsets of $X_+$ with zero aperture by
\begin{align*}
\mathcal{P}_0(X)&=\{Y\in\mathcal{P}(X_+):\alpha(Y)=0\}.\\
\mathcal{P}_0^\mathbf{d}(X)&=\{Y\in\mathcal{P}^{\mathbf{d}_+}(X_+):\alpha(Y)=0\}.\\
\mathcal{I}_0^\mathbf{d}(X)&=\{Y\in\mathcal{I}^{\mathbf{d}_+}(X_+):\alpha(Y)=0\}.
\end{align*}
Note we have a natural embedding of $X$ into $\mathcal{P}(X_+)$ given by
\[x\mapsto x_0=(\leq^{\mathbf{d}_+}\!\!(x,0))=\{(y,s)\in X_+:y\mathbf{d}x\leq s\}.\]

\begin{thm}\label{toppredomaincompletion}
If $X$ is $\mathbf{d}^\bullet_\circ$-continuous then
\begin{equation}\label{topdH=dH|}
\underline{\mathbf{d}_+^\mathcal{H}|_{\mathcal{P}_0^\mathbf{d}(X)}}\ =\ \mathbf{d}_{+\mathcal{H}}|_{\mathcal{P}_0^\mathbf{d}(X)}.
\end{equation}
Moreover, $\mathcal{P}_0^\mathbf{d}(X)$ is a $\mathbf{d}_+^\mathcal{H}{}^\bullet_\circ$-domain with $\mathbf{d}^\mathcal{H}_+{}^\bullet_\circ$-basis $(x_0)_{x\in X}$ and
\begin{equation}\label{topdHxdy}
x_0\mathbf{d}_+^\mathcal{H}y_0\ \leq\ x\mathbf{d}y.
\end{equation}
\end{thm}

\begin{proof}  As $X$ is $\mathbf{d}^\bullet_\circ$-continuous, $X_+$ is $\mathbf{d}_+{}^\bullet_\circ$-continuous, by \eqref{X+cts}, and hence $\mathbf{d}_+$-$\max$-continuous, by \autoref{ctscor} \eqref{ctscor1} and \eqref{d+PInterpolation}.  In particular, each $x_0\in\mathcal{P}(X_+)$ is $\mathbf{d}_+$-directed and also has zero aperture, as $\mathbf{0}\circ\mathbf{d}=\mathbf{0}$, so $x_0\in\mathcal{P}_0^\mathbf{d}(X)$.

For $\mathbf{d}_+^\mathcal{H}{}^\bullet_\circ$-continuity, take $W\in\mathcal{P}_0^\mathbf{d}(X)$, $Y,Z\subseteq X_+$ and $r>Y\mathbf{d}_+^\mathcal{H}W,Z\mathbf{d}_+^\mathcal{H}W$, so we have $(u,s),(v,t)\in W$ with $Y\mathbf{d}_+(u,s),Z\mathbf{d}_+(v,t)<r$.  As $\alpha(W)=0$, for any $\epsilon>0$, we have $(x,\epsilon')\in W$, for some $\epsilon\in(0,\epsilon)$.  As $W$ is $\mathbf{d}$-directed, we have $(w,\delta)\in W$ with $\{(u,s),(v,t),(x,\epsilon')\}\mathbf{d}_+(w,\delta)<\epsilon-\epsilon'$.  In particular,
\[\delta-\epsilon'\leq(x\mathbf{d}w-\epsilon'+\delta)_+=(x,\epsilon')\mathbf{d}_+(w,\delta)<\epsilon-\epsilon'\]
so $w_0\mathbf{d}_+^\mathcal{H}W\leq w_0\mathbf{d}_+(w,\delta)\leq\delta<\epsilon$.
Also $\mathbf{d}_+$-$\max$-continuity yields
\[(u,s)\mathbf{d}_+w_0=(u,s)\mathbf{d}_+(w,0)\leq(u,s)\mathbf{d}_+(w,\delta)<\epsilon,\]
so $Y\mathbf{d}_+^\mathcal{H}w_0<r+\epsilon$ and, likewise, $Z\mathbf{d}_+^\mathcal{H}w_0<r+\epsilon$.  Thus, as $r$ and $\epsilon$ were arbitrary,
\[\mathcal{F}\mathbf{d}_+^\mathcal{H}\circ\{x_0:x\in X\}\circ\Phi^{\mathbf{d}_+^\mathcal{H}}\leq\mathcal{F}\mathbf{d}_+^\mathcal{H},\]
i.e. $\mathcal{P}_0^\mathbf{d}(X)$ is $\mathbf{d}_+^\mathcal{H}{}^\bullet_\circ$-continuous with $\mathbf{d}^\mathcal{H}_+{}^\bullet_\circ$-basis $\{x_0:x\in X\}$.\footnote{Alternatively one could argue that, for any $Y\in\mathcal{I}_0^\mathbf{d}(X)$ with ($<^\mathbf{d}$-directed) $\mathbf{d}_+^\mathcal{H}{}_\bullet$-interior $Z$, $(x_0)_{(x,r)\in Z}$ is a $\mathbf{d}_+^\mathcal{H}$-Cauchy net with $\mathbf{d}_+^\mathcal{H}{}^\bullet_\circ$-limit $Y$.}

As $\mathbf{d}_+\leq\mathbf{d}_+\circ\underline{\mathbf{d}_+}=\mathbf{d}_+\circ\underline{\mathbf{d}}_+$, \autoref{Hausfunc} yields $\mathbf{d}_+^\mathcal{H}\leq\mathbf{d}_+^\mathcal{H}\circ\underline{\mathbf{d}}_{+\mathcal{H}}$ and hence
\begin{equation}\label{dH+s}
\underline{\mathbf{d}_+^\mathcal{H}}\leq\underline{\mathbf{d}}_{+\mathcal{H}}.
\end{equation}
For $Y\subseteq X_+$ let $Y^r=\{(y,s+r):(y,s)\in Y\}$ so $Y\underline{\mathbf{d}_+^\mathcal{H}}Y^r\leq Y\underline{\mathbf{d}}_{+\mathcal{H}}Y^r\leq r$ and $Y\mathbf{d}_+^\mathcal{H}Z<r$ implies $Y^r\leq^{\mathbf{d}_+^\mathcal{H}}Z$.  This shows that $\underline{\mathbf{d}_+^\mathcal{H}}\ \circ\leq^{\mathbf{d}_+^\mathcal{H}\mathcal{P}}\ \leq\ \mathbf{d}_+^\mathcal{H}\mathcal{P}$, which is the required interpolation condition for \cite[(11.5)]{Bice2019a}.  This means $\mathcal{P}^{\mathbf{d}_+}(X_+)$ is $\mathbf{d}_+^\mathcal{H}{}^\bullet_\circ$-complete, as we already know $\mathcal{P}^{\mathbf{d}_+}(X_+)$ is $\mathbf{d}_+^\mathcal{H}$-$\max$-complete, by \autoref{predomaincompletion}.

Now note that $\alpha(Y)=X_+\mathbf{0}_{+\mathcal{H}}Y$ so, as $\mathbf{0}_{+\mathcal{H}}$ is a distance with $\mathbf{0}_{+\mathcal{H}}\leq\mathbf{d}_{+\mathcal{H}}$,
\begin{equation}\label{alphatri}
\alpha(Y)\leq\alpha(Z)+Z\mathbf{d}_{+\mathcal{H}}Y\leq\alpha(Z)+Z\mathbf{d}_+^\mathcal{H}Y.
\end{equation}
This means any $\mathbf{d}_+^\mathcal{H}{}_\circ$-limit of a $\mathbf{d}_+^\mathcal{H}$-Cauchy net in $\mathcal{P}_0^\mathbf{d}(X)$ also has zero aperture.  Thus $\mathcal{P}_0^\mathbf{d}(X)$ is also $\mathbf{d}_+^\mathcal{H}{}^\bullet_\circ$-complete.

\begin{itemize}
\item[\eqref{topdH=dH|}] For any $Y,Z\in\mathcal{P}^\mathbf{d}_0(X)$, the proof of \eqref{dH=dH|} yields
\[Y\mathbf{d}_{+\mathcal{H}}Z=\sup_{(y,s)\in Y}(\leq^{\mathbf{d}_+}\!\!(y,s))\mathbf{d}_+^\mathcal{H}Z.\]
As $Y$ is $\mathbf{d}_+$-directed and $\alpha(Y)=0$, for any $\epsilon>0$, we can restrict to $s<\epsilon$,
\begin{align*}
Y\mathbf{d}_{+\mathcal{H}}Z&=\sup_{(y,s)\in Y,s<\epsilon}(\leq^{\mathbf{d}_+}\!\!(y,s))\mathbf{d}_+^\mathcal{H}Z\\
&\leq\sup_{(y,s)\in Y,s<\epsilon}y_0\mathbf{d}_+^\mathcal{H}Z\\
&\leq\sup_{(y,s)\in Y,s<\epsilon}(y_0\mathbf{d}_+^\mathcal{H}Z-y_0\mathbf{d}_+^\mathcal{H}Y+y_0\mathbf{d}_+(y,s)))\\
&\leq\sup_{(y,s)\in Y,s<\epsilon}(y_0\mathbf{d}_+^\mathcal{H}Z-y_0\mathbf{d}_+^\mathcal{H}Y+s)\\
&\leq\sup_{(y,s)\in Y,s<\epsilon}(y_0\mathbf{d}_+^\mathcal{H}Z-y_0\mathbf{d}_+^\mathcal{H}Y)+\epsilon.\\
&\leq Y(\underline{\mathbf{d}_+^\mathcal{H}|_{\mathcal{P}_0^\mathbf{d}(X)}})Z+\epsilon,
\end{align*}
as each $y_0\in\mathcal{P}_0^\mathbf{d}(X)$.  This and \eqref{dH+s} yields \eqref{topdH=dH|}, as
\[\mathbf{d}_{+\mathcal{H}}|_{\mathcal{P}_0^\mathbf{d}(X)}\leq\underline{\mathbf{d}_+^\mathcal{H}|_{\mathcal{P}_0^\mathbf{d}(X)}}\leq\underline{\mathbf{d}_+^\mathcal{H}}|_{\mathcal{P}_0^\mathbf{d}(X)}\leq\underline{\mathbf{d}}_{+\mathcal{H}}|_{\mathcal{P}_0^\mathbf{d}(X)}\leq\mathbf{d}_{+\mathcal{H}}|_{\mathcal{P}_0^\mathbf{d}(X)}.\]
\end{itemize}
By \eqref{de^H}, $\mathbf{d}_+^\mathcal{H}\leq\mathbf{d}_{+\mathcal{H}}\circ\mathbf{d}_+^\mathcal{H}$ so $\overline{\mathbf{d}_+^\mathcal{H}}\leq\mathbf{d}_{+\mathcal{H}}$ and hence, by \eqref{topdH=dH|},
\[\overline{\mathbf{d}_+^\mathcal{H}|_{\mathcal{P}_0^\mathbf{d}(X)}}\leq\overline{\mathbf{d}_+^\mathcal{H}}|_{\mathcal{P}_0^\mathbf{d}(X)}\leq\mathbf{d}_{+\mathcal{H}}|_{\mathcal{P}_0^\mathbf{d}(X)}=\underline{\mathbf{d}_+^\mathcal{H}|_{\mathcal{P}_0^\mathbf{d}(X)}}.\]
Thus $\mathcal{P}_0^\mathbf{d}(X)$ is a $\mathbf{d}_+^\mathcal{H}{}^\bullet_\circ$-domain.  Lastly, for \eqref{topdHxdy}, note that \eqref{dHxdy} yields
\[x_0\mathbf{d}_+^\mathcal{H}y_0\ \leq\ (x,0)\mathbf{d}_+(y,0)=x\mathbf{d}y.\qedhere\]
\end{proof}

\begin{cor}\label{toppdcomp}
The following are equivalent.
\begin{enumerate}
\item\label{toppdcompp} X is a $\mathbf{d}^\bullet_\circ$-predomain.
\item\label{toppdcompb} X is a $\mathbf{d}'{}^\bullet_\circ$-basis of a $\mathbf{d}'{}^\bullet_\circ$-domain $X'\supseteq X$ with $\mathbf{d}'|_X=\mathbf{d}$.
\end{enumerate}
\end{cor}

\begin{proof}\
\begin{itemize}
\item[\eqref{toppdcompp}$\Rightarrow$\eqref{toppdcompb}]  Assume \eqref{toppdcompp} and let $X'$ be the (disjoint) union of $X$ and $\mathcal{P}_0^\mathbf{d}(X)$.  Extend $\mathbf{d}_+^\mathcal{H}$ to $\mathbf{d}'$ on $X'$ by making each $x\in X$ $\mathbf{d}'$-equivalent to $x_0$.  By \autoref{toppredomaincompletion}, the only thing left to show is that the inequality in \eqref{topdHxdy} is an equality.  But $\overline{\mathbf{d}}\leq\underline{\mathbf{d}}$ implies $\overline{\mathbf{d}_+}\leq\overline{\mathbf{d}}_+\leq\underline{\mathbf{d}}_+=\underline{\mathbf{d}_+}$ so, by \eqref{dHxdy=},
\[x_0\mathbf{d}_+^\mathcal{H}y_0\geq(x,0)\mathbf{d}_+(y,0)=x\mathbf{d}y.\]

\item[\eqref{toppdcompb}$\Rightarrow$\eqref{toppdcompp}]  If $X\subseteq X'$ is a $\mathbf{d}'{}^\bullet_\circ$-basis and $\mathbf{d}=\mathbf{d}'|_X$ then $X$ is certainly $\mathbf{d}^\bullet_\circ$-continuous.  If $X'$ is also a $\mathbf{d}'{}^\bullet_\circ$-(pre)domain then $\overline{\mathbf{d}}=\overline{\mathbf{d}'}|_X\leq\underline{\mathbf{d}'}|_X=\underline{\mathbf{d}}$, by \autoref{reflexbasis}, i.e. $X$ is a $\mathbf{d}^\bullet_\circ$-predomain.\qedhere
\end{itemize}
\end{proof}

In particular, any hemimetric space $(X,\mathbf{d})$ has a Smyth completion $(X',\mathbf{d}')$, but there is no guarantee that $\mathbf{d}'$ will also be a hemimetric, i.e. $\leq^{\mathbf{d}'}$ may not be reflexive on the larger space $X'$.  On the other hand, $\mathbf{d}_{+\mathcal{H}}$ is always a hemimetric on $\mathcal{P}_0^\mathbf{d}(X)$, which is $\mathbf{d}_{+\mathcal{H}}{}^\circ_\circ$-complete by \autoref{Tdomaineqs} and \autoref{toppredomaincompletion}.  Indeed the hemimetric space $(\mathcal{P}_0^\mathbf{d}(X),\mathbf{d}_{+\mathcal{H}})$, or the equivalent quasimetric space $(\mathcal{I}_0^\mathbf{d}(X),\mathbf{d}_{+\mathcal{H}})$, is often called the \emph{Yoneda completion} of $X$.  In fact, by the following result and \eqref{topdH=dH|}, we see that $X$ has a hemimetric Smyth completion precisely when it coincides with the Yoneda completion.

\begin{thm}\label{ESmyth}
If $\mathbf{d}$ is a hemimetric, the following are equivalent.
\begin{enumerate}
\item\label{dNoth} $X$ is $\mathbf{d}$-Noetherian.
\item\label{Scomp} $(\mathcal{P}_0^\mathbf{d}(X),\mathbf{d}_{+\mathcal{H}})$ is Smyth complete.
\item\label{hemispace} $(\mathcal{P}_0^\mathbf{d}(X),\mathbf{d}_+^\mathcal{H})$ is a hemimetric space.
\item\label{EScomp} $(X,\mathbf{d})$ has a hemimetric Smyth completion.
\end{enumerate}
\end{thm}

\begin{proof}\
\begin{itemize}
\item[\eqref{EScomp}$\Rightarrow$\eqref{dNoth}]  We show that any $\mathbf{d}$-Cauchy $(x_n)$ in a Smyth complete hemimetric space $(X,\mathbf{d})$ is $\mathbf{d}^\mathrm{op}$-pre-Cauchy.  Indeed Smyth completeness yields $x=\mathbf{d}^\bullet_\circ$-$\lim x_n$ and then $\underline{\mathbf{d}}=\mathbf{d}$ and \cite[(8.13)]{Bice2019a} yield
\[\lim_j\lim_k x_k\mathbf{d}x_j=\lim_j x\mathbf{d}x_j=x\mathbf{d}x=0.\]

\item[\eqref{dNoth}$\Rightarrow$\eqref{hemispace}]  If every $\mathbf{d}$-Cauchy net in $X$ is $\mathbf{d}^\mathrm{op}$-Cauchy then every $\mathbf{d}_+$-Cauchy net in $X_+$ is ($\mathbf{d}_+)^\mathrm{op}$-Cauchy (as in the alternative proof of \eqref{Bcomp}).  Thus, by \autoref{dNothdH}, $\mathbf{d}_+^\mathcal{H}$ is a hemimetric on $\mathcal{P}^{\mathbf{d}_+}(X_+)$ and hence $\mathcal{P}_0^\mathbf{d}(X)$.

\item[\eqref{hemispace}$\Rightarrow$\eqref{Scomp}]  By \autoref{toppredomaincompletion}, $\mathcal{P}_0^\mathbf{d}(X)$ is $\mathbf{d}_+^\mathcal{H}{}^\bullet_\circ$-complete.  If $\mathbf{d}_+^\mathcal{H}$ is a hemimetric then
\[\mathbf{d}_+^\mathcal{H}=\underline{\mathbf{d}_+^\mathcal{H}}\leq\underline{\mathbf{d}}_{+\mathcal{H}}=\mathbf{d}_{+\mathcal{H}}\leq\mathbf{d}_+^\mathcal{H},\]
by \eqref{dH} and \eqref{dH+s}, so $\mathcal{P}_0^\mathbf{d}(X)$ is $\mathbf{d}_{+\mathcal{H}}{}^\bullet_\circ$-complete.

\item[\eqref{Scomp}$\Rightarrow$\eqref{hemispace}]  Say $\mathcal{P}_0^\mathbf{d}(X)$ is $\mathbf{d}_{+\mathcal{H}}{}^\bullet_\circ$-complete.  As $\mathbf{d}_{+\mathcal{H}}$ is a hemimetric on $\mathcal{P}^\mathbf{d}_0(X)$, this means $\mathcal{P}^\mathbf{d}_0(X)$ is a $\mathbf{d}_{+\mathcal{H}}{}^\bullet_\circ$-domain with $\mathbf{d}_{+\mathcal{H}}=\underline{\mathbf{d}_{+\mathcal{H}}}$ so \autoref{Tdomaineqs} yields
\[\mathbf{d}_{+\mathcal{H}}|_{\mathcal{P}_0^\mathbf{d}(X)}={}^\circ_\circ(\mathbf{d}_{+\mathcal{H}}|_{\mathcal{P}_0^\mathbf{d}(X)}).\]
But $\mathcal{P}_0^\mathbf{d}(X)$ is also a $\mathbf{d}_+^\mathcal{H}{}^\bullet_\circ$-domain with $\underline{\mathbf{d}_+^\mathcal{H}|_{\mathcal{P}_0^\mathbf{d}(X)}}\ =\ \mathbf{d}_{+\mathcal{H}}|_{\mathcal{P}_0^\mathbf{d}(X)}$, by \autoref{toppredomaincompletion}, so again by \autoref{Tdomaineqs}, $\mathbf{d}_+^\mathcal{H}|_{\mathcal{P}_0^\mathbf{d}(X)}={}^\circ_\circ(\mathbf{d}_{+\mathcal{H}}|_{\mathcal{P}_0^\mathbf{d}(X)})=\mathbf{d}_{+\mathcal{H}}|_{\mathcal{P}_0^\mathbf{d}(X)}$.

\item[\eqref{hemispace}$\Rightarrow$\eqref{EScomp}]  If $\mathbf{d}_+^\mathcal{H}$ is a hemimetric on $\mathcal{P}_0^\mathbf{d}(X)$ then $(X',\mathbf{d}')$ in the proof of \autoref{toppdcomp} \eqref{toppdcompp}$\Rightarrow$\eqref{toppdcompb} is a hemimetric Smyth completion of $X$.\qedhere
\end{itemize}
\end{proof}

Finally, as in \autoref{predomainuniversality}, we see that $\mathbf{d}^\bullet_\circ$-completions are unique.  Indeed, the following is saying that $\mathcal{I}_0^\mathbf{d}(B)$ is universal among $\mathbf{d}^\bullet_\circ$-predomain extensions of $B$, and unique among $\mathbf{d}^\bullet_\circ$-domain extensions, up to isometry (and $\mathbf{d}$-equivalence).

\begin{thm}\label{toppredomainuniversality}
If $\mathbf{d}$ is a distance and $B$ is a $\mathbf{d}^\bullet_\circ$-basis of $\mathbf{d}^\bullet_\circ$-predomain $X$,
\begin{equation}\label{topisocomp}
x\mapsto x_0\cap B_+
\end{equation}
is an isometry to $\mathbf{d}_+^\mathcal{H}{}^\bullet_\circ$-domain $\mathcal{I}_0^\mathbf{d}(B)$, which is onto iff $X$ is a $\mathbf{d}^\bullet_\circ$-domain.
\end{thm}

\begin{proof}
As $B$ is a $\mathbf{d}^\bullet_\circ$-basis, $x_0\cap B_+\in\mathcal{I}_0^\mathbf{d}(B)$, for all $x\in X$.  Every $Y\in\mathcal{P}_0^\mathbf{d}(B)$ is $\mathbf{d}_+^\mathcal{H}$-equivalent to $I_Y=\overline{\mathbf{d}}_+^\bullet$-$\mathrm{cl}(Y)\in\mathcal{I}_0^\mathbf{d}(B)$ so, by \autoref{toppredomaincompletion}, $\mathcal{I}_0^\mathbf{d}(B)$ is a $\mathbf{d}_+^\mathcal{H}{}^\bullet_\circ$-domain and
\[(x_0\cap B_+)\mathbf{d}_+^\mathcal{H}(y_0\cap B_+)\ =\ x\mathbf{d}y,\]
i.e. \eqref{topisocomp} is an isometry.  Also, like in \eqref{Cideal}, for any $\mathbf{d}$-Cauchy $(x_\lambda)\subseteq X$, take
\[I=\{(x,r):x\in B\text{ and }x\mathbf{d}(x_\lambda)\leq r\},\]
so $I\in\mathcal{I}_0^\mathbf{d}(B)$ (and every $I\in \mathcal{I}_0^\mathbf{d}(B)$ is of this form).  Also, as $B$ is a $\mathbf{d}^\bullet_\circ$-basis,
\[x_\lambda\barrowc x\quad\Leftrightarrow\quad\text{$(x,0)=\mathbf{d}_+$-$\max I$}\quad\Leftrightarrow\quad\mathbf{d}_+(x,0)=\mathbf{d}I\quad\Leftrightarrow\quad(x_0\cap B_+)=I,\]
so \eqref{topisocomp} is onto iff $X$ is $\mathbf{d}{}^\bullet_\circ$-complete and hence a $\mathbf{d}^\bullet_\circ$-domain.
\end{proof}

\bibliography{maths}{}
\bibliographystyle{alphaurl}

\end{document}